\documentclass[reqno]{amsart}

\usepackage[utf8]{inputenc}
\usepackage[T1]{fontenc}
\usepackage{lmodern}

\usepackage[english]{babel}

\usepackage{amsmath}
\usepackage{amsfonts}
\usepackage{amssymb}
\usepackage{amsthm}

\usepackage{hyperref}

\theoremstyle{plain}
	\newtheorem{theorem}{Theorem}
	\newtheorem{proposition}[theorem]{Proposition}
	\newtheorem{lemma}[theorem]{Lemma}
	\newtheorem{corollary}[theorem]{Corollary}
\theoremstyle{definition}
	\newtheorem{definition}[theorem]{Definition}
	\newtheorem{notation}[theorem]{Notation}
\theoremstyle{remark}
	\newtheorem{remark}[theorem]{Remark}

\newcommand{\TrSig}{\sigma}
\newcommand{\InSig}{\tilde{\sigma}}
\newcommand{\SigSet}{\Sigma}
\newcommand{\InSigSet}{\tilde{\Sigma}}
\newcommand{\MulSet}{\mathsf{K}}
\newcommand{\ParSet}{\mathsf{X}}
\newcommand{\InParSet}{\tilde{\mathsf{X}}}
\newcommand{\De}{\mathcal{D}_e}
\newcommand{\Dgr}{\mathcal{D}_{\mathrm{gr}}}

\newcommand{\ZZ}{\mathbb{Z}}
\newcommand{\RR}{\mathbb{R}}
\newcommand{\CC}{\mathbb{C}}
\newcommand{\HH}{\mathbb{H}}
\newcommand{\FF}{\mathbb{F}}
\newcommand{\KK}{\mathbb{K}}
\newcommand{\cL}{\mathcal{L}}
\newcommand{\cR}{\mathcal{R}}
\newcommand{\cD}{\mathcal{D}}
\newcommand{\cV}{\mathcal{V}}
\newcommand{\supp}{\mathrm{supp}\,}
\newcommand{\rad}{\mathrm{rad}}
\newcommand{\Int}{\mathrm{Int}}
\newcommand{\End}{\mathrm{End}}
\newcommand{\Arf}{\mathrm{Arf}}

\begin{document}

\title[Gradings on classical central simple real Lie algebras]{Gradings on classical central simple\\ real Lie algebras}

\author[Y. Bahturin]{Yuri Bahturin${}^{\star}$}
\address{Department of Mathematics and Statistics,
 Memorial University of Newfoundland,
 St. John's, NL, A1C5S7, Canada}
\email{bahturin@mun.ca}
\thanks{${}^{\star}$supported by Discovery Grant 227060-2014 
of the Natural Sciences and Engineering Research Council (NSERC) of Canada}

\author[M. Kochetov]{Mikhail Kochetov${}^{\star\star}$}
\address{Department of Mathematics and Statistics,
 Memorial University of Newfoundland,
 St. John's, NL, A1C5S7, Canada}
\email{mikhail@mun.ca}
\thanks{${}^{\star\star}$supported by Discovery Grant 341792-2013
of the Natural Sciences and Engineering Research Council (NSERC) of Canada 
and a grant for visiting scientists by Instituto Universitario de Matem\'aticas y Aplicaciones, 
University of Zaragoza}

\author[A. Rodrigo-Escudero]{Adri\'an Rodrigo-Escudero${}^\dagger$}
\address{Departamento de Matem\'aticas e
Instituto Uni\-ver\-si\-ta\-rio de Matem\'aticas y Apli\-ca\-cio\-nes,
Universidad de Zaragoza, 50009 Zaragoza, Spain}
\email{adrian.rodrigo.escudero@gmail.com}
\thanks{${}^\dagger$supported by
a doctoral grant of the Diputaci\'on General de Arag\'on;
a three-month stay
with Atlantic Algebra Centre
at Memorial University of Newfoundland
was supported by the
Fundaci\'on Bancaria Ibercaja and Fundaci\'on CAI
(reference number CB 4/15)
and by the NSERC of Canada
(Discovery Grant 227060-2014);
also supported by
the Spanish Mi\-nis\-te\-rio de Econom\'ia y Competitividad
and Fondo Europeo de De\-sa\-rro\-llo Regional FEDER
(MTM2013-45588-C3-2-P)
and by the Diputaci\'on General de Arag\'on
and Fondo Social Europeo
(Grupo de Investigaci\'on de \'Algebra)}

\subjclass[2010]{Primary 17B70; Secondary 16W50, 16W10}

\keywords{Graded algebra, graded module, classical simple Lie algebra, real algebra, classification}

\date{}

\begin{abstract}
For any abelian group $G$, we classify up to isomorphism all $G$-gradings on the classical central simple Lie algebras,
except those of type $D_4$, over the field of real numbers (or any real closed field).
\end{abstract}

\maketitle

\section{Introduction}

Let $\cR$ be an algebra (not necessarily associative) over a field and let $G$ be a group. 
A \emph{$G$-grading} on $\cR$ is a vector space decomposition $\Gamma:\;\cR=\bigoplus_{g\in G}\cR_g$ 
such that $\cR_g\cR_h\subseteq\cR_{gh}$ for all $g,h\in G$. 
The nonzero elements $x\in\cR_g$ are said to be \emph{homogeneous of degree} $g$, which can be written as $\deg x=g$, 
and the \emph{support} of the grading $\Gamma$ (or of the graded algebra $\cR$) is the set $\{g\in G\;|\;\cR_g\neq 0\}$. 
The algebra $\cR$ may have some additional structure, for example, an involution $\varphi$, 
in which case $\Gamma$ is required to respect this structure: $\varphi(\cR_g)=\cR_g$ for all $g\in G$. 

Group gradings have been extensively studied for many types of algebras --- associative, Lie, Jordan, composition, etc. 
(see e.g. the recent monograph \cite{EK13} and the references therein). In the case of gradings on simple Lie algebras, 
the support generates an abelian subgroup of $G$, so it is no loss of generality to assume $G$ abelian. 
We will do so for all gradings considered in this paper.
 
The classification of fine gradings (up to equivalence) on all finite-dimensional simple Lie algebras 
over an algebraically closed field of characteristic $0$ has recently been completed by the efforts of many authors: 
see \cite{E10}, \cite[Chapters 3--6]{EK13}, \cite{DV16},
\cite{YuExc} and \cite{E14}.
The classification of all $G$-gradings (up to isomorphism) 
is also known for these algebras, except for types $E_6$, $E_7$ and $E_8$, over an algebraically closed field of 
characteristic different from $2$:
see \cite{BK10},
\cite[Chapters 3--6]{EK13} and \cite{EK15}.

On the other hand, group gradings on algebras over the field of real numbers have not yet been sufficiently studied. 
Fine gradings on real forms of the classical simple complex Lie algebras except $D_4$ were described in \cite{HPP3}, 
but the equivalence problem remains open. 
Fine gradings on real forms of the simple complex Lie algebras of types $G_2$ and $F_4$
were classified up to equivalence in \cite{CDM10}; some partial results were obtained for type $E_6$ in \cite{DG16}. 

Here we will classify up to isomorphism all $G$-gradings on real forms of the classical simple complex Lie algebras 
except $D_4$.
We do not use the topology of $\RR$, so in all of our results $\RR$ can be replaced by an arbitrary real closed field.
We will follow the approach that has already been used over algebraically closed fields \cite{BK10,EK13}: embed our Lie algebra $\cL$ into an associative algebra $\cR$ with involution $\varphi$ 
and transfer the classification problem for gradings from $\cL$ to $(\cR,\varphi)$ using automorphism group schemes. 
Specifically:

\begin{itemize}
\item $\cR=M_n(\RR)$, with $\varphi$ orthogonal (that is, given by a symmetric bilinear form) and $n$ odd, 
yield all real forms of series $B$ by taking $\cL=\mathrm{Skew}(\cR,\varphi) := \{x\in\cR \mid \varphi(x)=-x\}$;
\item $\cR=M_n(\RR)$, with $\varphi$ symplectic (that is, given by a skew-symmetric bilinear form over $\RR$, so $n$ must be even), 
and $\cR=M_n(\HH)$, with $\varphi$ symplectic (that is, given by a hermitian form over $\HH$), 
yield all real forms of series $C$ by taking $\cL=\mathrm{Skew}(\cR,\varphi)$;
\item $\cR=M_n(\RR)$, with $\varphi$ orthogonal (that is, given by a symmetric bilinear form over $\RR$) and $n$ even, and 
$\cR=M_n(\HH)$, with $\varphi$ orthogonal (that is, given by a skew-hermitian form over $\HH$), 
yield all real forms of series $D$ by taking $\cL=\mathrm{Skew}(\cR,\varphi)$;
\item $\cR=M_n(\CC)$, $M_n(\RR)\times M_n(\RR)$ and $M_n(\HH)\times M_n(\HH)$, with $\varphi$ of the second kind 
(that is, nontrivial on the center of $\cR$), 
yield all real forms of series $A$ by taking $\cL=\mathrm{Skew}(\cR,\varphi)'$, where prime denotes 
the derived Lie algebra.
\end{itemize}

With the exception of the last two, the above associative algebras $\cR$ are simple, so for any $G$-grading on $\cR$ we have 
$\cR\cong\mathrm{End}_\cD(\cV)$ as a graded algebra, where $\cD$ is a $G$-graded associative algebra
which is a \emph{graded division algebra} (that is, all nonzero homogeneous elements are invertible),
and $\cV$ is a graded right $\cD$-module which is finite-dimensional over $\cD$
(that is, has a finite basis consisting of homogeneous elements) --- see details in Section \ref{se:assoc_description}.
Note that the graded algebra $\cD$ is determined up to isomorphism and the graded module $\cV$ up to isomorphism and shift of grading.
By a \emph{shift of grading} we mean replacing $\cV$ by $\cV^{[g]}$ for some $g\in G$, where $\cV^{[g]}=\cV$ as a module, but with 
the elements of $\cV_h$ now having degree $gh$, for all $h\in G$. 
Over an algebraically closed field $\FF$, the identity component $\cD_e$ ($e$ being the identity element of $G$) 
must be $\FF$, so the involution on $\cR$ can be studied, as was done in \cite{BZ07,BK10}, 
using the tensor product decomposition $\cR\cong\mathrm{End}_{\FF}(V)\otimes_\FF\cD$ 
where $V$ is the $\FF$-span of a homogeneous $\cD$-basis of $\cV$.
Here we will follow a graded version of the classical approach, as was done in \cite{E10,EK13}, 
and write the involution on $\cR$ in terms of an involution on $\cD$ and a sesquilinear form on $\cV$.

As to the algebras $\cR=M_n(\RR)\times M_n(\RR)$ and $\cR=M_n(\HH)\times M_n(\HH)$, we have two types of $G$-gradings: 
if the center $Z(\cR)$ is trivially graded then we will say that the grading is of \emph{Type~I}, and otherwise of \emph{Type~II}. 
In the case of a Type~II grading, $\cR$ is \emph{graded-simple} (that is, simple as a graded algebra: it has no nonzero proper graded ideals) 
and can apply the same approach as above.
In the case of a Type~I grading, we will use alternative models: $\cR=M_n(\RR)$ and $\cR=M_n(\HH)$, where $\cL=\cR'$ 
(traceless matrices). In these models, there is no involution on $\cR$, so they are treated separately.
Note that we also have Type~I and Type~II gradings for $M_n(\CC)$, but we do not have to use a different approach 
for Type~I gradings.

Finite-dimensional graded division algebras over $\RR$ that are simple as ungraded algebras
and have an abelian grading group have recently been classified
in \cite{R16}, both up to isomorphism and up to equivalence,
and independently in \cite{BZ16}, up to equivalence
(but note that one of the equivalence classes was overlooked).
A classification up to equivalence has been obtained in \cite{BZpr} without assuming simplicity.

For our purposes, we will also need some information about (degree-preserving) involutions 
on a graded division algebra $\cD$ that is isomorphic to $M_\ell(\RR)$, $M_\ell(\CC)$, $M_\ell(\HH)$, 
$M_\ell(\RR)\times M_\ell(\RR)$ or $M_\ell(\HH)\times M_\ell(\HH)$ as an ungraded algebra, 
where $\ell$ is a divisor of $n$.
A classification of such graded division algebras with involution is given in \cite{BRpr}.

The paper is structured as follows. In Section \ref{se:assoc_description}, we establish some properties of 
(associative) graded division algebras with involution over $\RR$, which may be of independent interest. 
We do not restrict ourselves to the finite-dimensional case:
we only assume that $\De$ is finite-dimensional. About the involution, we assume that the only symmetric elements 
in $Z(\cD)_e$ are the scalars (in other words, $\cD$ is central as a graded algebra with involution). 
The main result (Theorem \ref{th:struct}) is a structure theorem for 
central graded-simple associative algebras with involution $(\cR,\varphi)$ over $\RR$,
under certain finiteness assumptions.
The structure of $(\cR,\varphi)$ is described by a graded division algebra with involution $(\cD,\varphi_0)$, 
an element $g_0\in G$, a sign $\delta\in\{\pm 1\}$, and two functions, $\kappa:G/T\to\ZZ_{\ge 0}$ and $\TrSig:G/T\to\ZZ$,
satisfying certain conditions (see Definitions \ref{def:MultFunct} and \ref{def:SignFunct}),
where $T$ is the support of $\cD$ (a subgroup of $G$).

In Section \ref{se:assoc_classification}, we solve the isomorphism problem for graded algebras with involution 
determined by these data (Theorem \ref{th:iso}), under a certain technical restriction in the case 
$\CC\cong\De\subseteq Z(\cD)$. In particular, if $\cD$ is isomorphic to 
$M_\ell(\RR)$, $M_\ell(\CC)$, $M_\ell(\HH)$, $M_\ell(\RR)\times M_\ell(\RR)$ or $M_\ell(\HH)\times M_\ell(\HH)$ 
as an ungraded algebra and $\varphi_0$ is of the second kind in the cases with $Z(\cD)\ne\RR$,
then the classification boils down to orbits for an action of $G$ on triples $(g_0,\kappa,\TrSig)$ 
(see Corollaries \ref{cor:1}, \ref{cor:2} and \ref{cor:3}). 
This action is associated with the shift of grading mentioned above.

In Section \ref{se:Lie}, we transfer the results to classical Lie algebras: 
Theorems \ref{th:AI_RH} and \ref{th:AI_C} for Type~I (that is, inner) gradings on real forms of series $A$, 
Theorems \ref{th:AII_RH} and \ref{th:AII_C} for Type~II (that is, outer) gradings on real forms of series $A$, 
and Theorems \ref{th:B}, \ref{th:C} and \ref{th:D} for series $B$, $C$ and $D$, respectively. 
In particular, we determine which of the real forms arises from given data describing $(\cR,\varphi)$.

Finally, in the Appendix, we show, under a certain condition, how a $G$-grading of Type II 
on one of the special linear Lie algebras $\mathrm{Skew}(\cR,\varphi)'$, 
where $\cR=M_n(\Delta)\times M_n(\Delta)$ and $\Delta\in\{\RR,\HH\}$, 
can be realized in terms of the model $\mathfrak{sl}_n(\Delta)$. 
The condition is satisfied, for example, if $G$ is an elementary $2$-group. 
We also consider Type II gradings on simple Lie algebras of series A over an algebraically closed field $\FF$, 
$\mathrm{char}\,\FF\ne 2$, which were treated in \cite{BK10,EK13} in terms of the model $\mathfrak{sl}_n(\FF)$ 
(modulo the $1$-dimensional center if $\mathrm{char}\,\FF$ divides $n$).
We show how the model $\mathrm{Skew}(\cR,\varphi)'$ (modulo the center), 
where $\cR=M_n(\FF)\times M_n(\FF)$, gives a different --- and simpler --- parametrization of Type II gradings.

For any abelian group $G$, we will use the notation $G_{[2]}:=\{g\in G \mid g^2=e\}$ 
and $G^{[2]}:=\{g^2 \mid g\in G\}$.

\section{Involutions on graded-simple real associative algebras}\label{se:assoc_description}

\subsection{Graded-simple real associative algebras}

Let $\mathcal{R}$ be a $G$-graded real associative algebra
that is graded-simple and satisfies the descending chain condition on graded left ideals.
By \cite[Theorem 2.6]{EK13} (compare with \cite[Theorem 2.10.10]{NVO}), $\mathcal{R}$ is isomorphic to 
$\mathrm{End}_{\mathcal{D}} (\mathcal{V}) $,
where $\mathcal{D}$ is a $G$-graded real associative algebra
which is a graded division algebra,
and $\mathcal{V}$ is a graded right $\mathcal{D}$-module
which is finite-dimensional over $\mathcal{D}$.
For an invertible $d\in\cD$,
we will denote by $\mathrm{Int}(d)$ the corresponding inner automorphism of $\cD$:
$ \mathrm{Int}(d)(x):=dxd^{-1} $.

Let $ T \subseteq G $ be the support of $\mathcal{D}$.
We fix a transversal for the subgroup $T$ in $G$,
so we have a section $ \xi : G/T \rightarrow G $
(which is not necessarily a group homomorphism)
of the quotient map $ \pi : G \rightarrow G/T $.
We can write $ \mathcal{V} = \mathcal{V}_1 \oplus \ldots \oplus \mathcal{V}_s $ 
where $\mathcal{V}_i$ are the isotypic components
of the graded right $\mathcal{D}$-module $\mathcal{V}$.
Each of these components is determined by its support,
which is a coset $ x_i \in G/T $, and its $\mathcal{D}$-dimension $k_i$.
Thus, $\mathcal{V}$ is determined by the multiset
$ \kappa : G/T \rightarrow \mathbb{Z}_{\geq 0} $
whose underlying set is $ \{ x_1 , \ldots , x_s \} $
and the multiplicity of $x_i$ is $ \kappa(x_i) = k_i $.
Let $ g_i := \xi(x_i) $.
Then $\mathcal{V}_{g_i}$ is a (right) vector space over the division algebra $\De$
and $ \mathcal{V}_i = \mathcal{V}_{g_i} \otimes_{\De} \mathcal{D} $;
hence $ \mathcal{V} = V \otimes_{\De} \mathcal{D} $,
with $ V := \mathcal{V}_{g_1} \oplus \ldots \oplus \mathcal{V}_{g_s} $.
Picking a basis $\mathcal{B}_i$ in each $\cV_{g_i}$,
we may identify $\mathrm{End}_\cD(\cV)$ with the algebra of matrices $M_k(\cD)$,
where $k=\vert\kappa\vert:=k_1+\cdots+k_s$.
The matrices are partitioned into $s^2$ blocks according to
the partition of the basis $\mathcal{B}_1\cup\ldots\cup\mathcal{B}_s$,
and the $G$-grading is given by
\begin{equation}\label{eq:MatrGrad}
\deg(E\otimes d)=g_i t g_j^{-1}
\end{equation}
for any matrix unit $E$ in the $(i,j)$-th block and any $d\in\cD_t$.
Here we are using the Kronecker product to identify
$M_k(\cD)$ with $M_k(\RR)\otimes_{\RR}\cD$.

Therefore, the data that determines the graded algebra $\cR$ is $(\cD,\kappa)$.
Conversely, any pair $(\cD,\kappa)$,
where $\cD$ is a graded-division real associative algebra,
and $ \kappa : G/T \rightarrow \mathbb{Z}_{ \geq 0 } $ has finite support,
determines, by means of Equation \eqref{eq:MatrGrad},
a graded-simple real associative algebra $M_k(\cD)$
that satisfies the descending chain condition on graded left ideals
(see \cite[Section 2.1]{EK13}).

We will assume throughout that the identity component $\cR_e$ is finite-dimensional.
From the above matrix description it is clear that
$\cR_e$ is a finite-dimensional vector space over the division algebra $\De$
(as $\cR_e$ consists of the diagonal blocks with entries in $\De$).
Thus, our assumption is tantamount to $\De$ having finite dimension,
which makes it isomorphic to $\mathbb{R}$, $\mathbb{C}$ or $\mathbb{H}$.

\subsection{Properties of the graded-division algebra}

We will now record for future reference some elementary properties of $\cD$.
First of all, for each $ t \in T $, pick $ 0 \neq X_t \in \mathcal{D}_t $.
Then $\cD_t=\De X_t=X_t\De$.
If $ \De \cong \mathbb{H} $, we can take $ X_t \in C_{\mathcal{D}}(\De) $.
Indeed, $\mathrm{Int}(X_t)$ restricts to an automorphism of $\De$,
which is inner since $\De$ is central simple.
Replacing $X_t$ by $X_tq$ for a suitable $q\in\De$, we get the result.
It also follows from the Double Centralizer Theorem
(see for instance \cite[Theorem 4.7]{Jac89}),
which says that we have an isomorphism of algebras
\[ \mathcal{D} \cong \De \otimes_{\mathbb{R}} C_{\mathcal{D}} (\De) \]
given by multiplication in $\cD$.
We will always assume that $X_t$ are chosen in $C_{\mathcal{D}}(\De)$ in the case $\De\cong \mathbb{H}$.
Of course, this is automatic in the case $\De= \mathbb{R}$.
In these cases, define $\beta:T\times T\to Z(\De)^\times = \mathbb{R}^\times $ by
\begin{equation}\label{def:beta}
X_sX_t=\beta(s,t) X_t X_s.
\end{equation}
Clearly, this is an alternating bicharacter that does not depend on the scaling of $X_t$
and is therefore an invariant of the graded algebra $\cD$.

As to the remaining case $\De\cong \mathbb{C}$, let
\[ K := \{ t \in T \mid \mathrm{Int} (X_t) \vert_{\De} = \mathrm{id}_{\De} \}. \]
This is a subgroup in $T$ of index $1$ or $2$,
and it does not depend on the choice of $X_t$.
Equation \eqref{def:beta} gives a well defined alternating bicharacter $\beta:K\times K\to\De^\times\cong\mathbb{C}^\times$.
It is convenient to set $K:=T$ in the cases $ \De = \mathbb{R} $ and $ \De \cong \mathbb{H} $, so in all cases 
we have $\beta:K\times K\to Z(\De)^\times$.

\begin{lemma}[{\cite[Lemma 3.3]{E10}}]\label{lem:InnAut}
Let $\alpha$ be an inner automorphism of $\cD$ that preserves degree.
Then, there exists a nonzero homogeneous $ d \in \cD $ such that $ \alpha = \Int(d) $.
\end{lemma}

\begin{proof}
By definition, there exists $ d' \in \cD^{\times} $ such that $ \alpha = \Int(d') $,
that is, $ \alpha(x) d' = d' x $ for all $ x \in \cD $.
Write $ d' = d_1 + \cdots + d_n $ where the $d_i$ are
nonzero homogeneous elements of pairwise distinct degrees $h_i$.
As $\alpha$ preserves dregree,
if $ x \in \cD $ is homogeneous of degree $h$, so is $\alpha(x)$.
Since $G$ is abelian,
if we consider the terms of degree $ h h_1 $ in $ \alpha(x) d' = d' x $,
we get $ \alpha(x) d_1 = d_1 x $.
But $d_1$ is invertible because it is homogeneous,
so $ \alpha(x) = d_1 x d_1^{-1} $ for all $ x \in \cD_h $ and $ h \in G $.
\end{proof}

\begin{lemma}\label{lem:cocycle}
The group of automorphisms of the graded algebra $\cD$ that act trivially on $\De$
is isomorphic to $ Z^1 (T, Z(\De)^{\times} ) $,
where $t\in T$ acts on $Z(\De)^{\times}$ on the right by
$ z^t = X_t^{-1} z X_t $ for all $ z \in Z(\De)^{\times} $.
\end{lemma}

\begin{proof}
Let $\psi_0$ be such an automorphism.
Then $ \psi_0(X_t) = X_t \nu(t) $, with $ \nu(t) \in \De^{\times} $,
and the condition $\psi_0\vert_{\De}=\mathrm{id}_{\De}$ 
implies $\Int(X_t)\vert_{\De}=\Int(\psi_0(X_t))\vert_{\De}$, hence $\nu(t)\in C_\cD(\De)$.
Therefore, we have $ \nu : T \rightarrow Z(\De)^{\times} $.
Applying $\psi_0$ to both sides of the equation $ X_s X_t = X_{st} \lambda $
($ \lambda \in \De^{\times} $),
we get $ X_s \nu(s) X_t \nu(t) = X_{st} \nu(st) \lambda $, so
\[ \nu(s)^t \nu(t) = \nu(st) \]
where the action is trivial for $ \De = \mathbb{R} $ and $ \De \cong \mathbb{H} $,
but may be nontrivial for $ \De \cong \mathbb{C} $ (recall $K$).
In other words, $\nu$ is a (right) $1$-cocycle.
The converse is also clear. Finally, if $1$-cocycles $\nu$ and $\nu'$ correspond
to $\psi_0$ and $\psi_0'$, respectively, then the product $\nu\nu'$ corresponds to
the composition $\psi_0\psi_0'$.
\end{proof}

If $ \De = \mathbb{R} $, $ \De \cong \mathbb{H} $
or $ \mathbb{C} \cong \De \subseteq Z(\cD) $ (that is, $K=T$),
then the action is trivial, so
$ Z^1 (T, Z(\De)^{\times} ) = \mathrm{Hom} (T, Z(\De)^{\times} ) $.

\begin{lemma}\label{lem:sses}
If $ \mathbb{C} \cong \De \not\subseteq Z(\cD) $,
then there is a split short exact sequence:
\[ 1 \longrightarrow U \longrightarrow
Z^1 (T, \De^{\times} )
\overset{\mathrm{res}}{\longrightarrow}
\mathrm{Hom}^+ (K, \mathbb{R}^{\times} ) \longrightarrow 1 \]
where $ \mathrm{Hom}^+ (K, \mathbb{R}^{\times} ) :=
\{ \chi \in \mathrm{Hom} (K, \mathbb{R}^{\times} ) \mid \chi(t^2) > 0 , \, \forall t \in T \} $,
$U$ is the unit circle in $\mathbb{C}$,
and $\mathrm{res}$ is the restriction map.
\end{lemma}

\begin{proof}
Let $ \nu \in Z^1 (T, \De^{\times} ) $ and pick $ t \in T \setminus K $.
For all $ s \in K $, we have
$ \overline{\nu(s)} \nu(t) = \nu(st) = \nu(ts) = \nu(t) \nu(s) $,
so $ \nu(s) \in \mathbb{R}^{\times} $.
For any $t\in T$,
we have $ \nu(t^2) = \nu(t)^t\nu(t)=\vert\nu(t)\vert^2 > 0 $,
since $\nu(t)^t=\nu(t)\in\RR$ if $ t \in K $
and $\nu(t)^t=\overline{\nu(t)}$ if $ t \in T \setminus K $.
Thus, the above restriction map is well defined.

Now take $\nu$ in the kernel of the restriction map,
that is, such that $ \nu(s) = 1 $ for all $ s \in K $.
If $ t \in T \setminus K $,
then $ \nu(ts) = \nu(t)^s \nu(s) = \nu(t) $,
so $\nu$ is constant on $ T \setminus K $.
Moreover, if $ t \in T \setminus K $ then $ t^2 \in K $,
so $ 1 = \nu(t^2) = \vert \nu(t) \vert^2 $.
Conversely, for any $ z \in U $, the map $\nu$,
defined by $ \nu(t) = 1 $ if $ t \in K $ and $ \nu(t) = z $ if $ t \in T \setminus K $,
is a $1$-cocycle.
Therefore, the kernel is isomorphic to $U$.

Finally, to construct a section of the restriction map,
we fix $ t \in T \setminus K $.
Then, for any $ \chi \in \mathrm{Hom}^+ (K, \mathbb{R}^{\times} ) $
define $ \nu_{\chi}(s) = \chi(s) $ and
$ \nu_{\chi}(ts) = \sqrt{\chi(t^2)} \chi(s) $ for all $ s \in K $.
It is easy to check that $\nu_{\chi}$ is a $1$-cocycle,
and the mapping $\chi\mapsto\nu_{\chi}$ is
a homomorphism of groups and a section of the restriction map.
This proves that $\mathrm{res}$ is surjective
and the above sequence splits (noncanonically).
\end{proof}

\begin{lemma}\label{lem:suprad}
The center $Z(\cD)$ is a graded subalgebra with support
\[ \mathrm{rad}\beta := \{ s \in K \mid \beta(s,t)=1 , \, \forall t \in K \}. \]
Except possibly in the case $ \mathbb{C} \cong \De \subseteq Z(\cD) $,
$\beta$ takes values in $\mathbb{R}^{\times}$.
\end{lemma}

\begin{proof}
As $G$ is abelian, the center $Z(\mathcal{D})$ is a graded subalgebra,
because $ x = \sum_{u \in G} x_u \in Z(\mathcal{D}) $ if and only if
$ x y = y x $ for all $ y \in \mathcal{D}_v $ and $v \in G$,
if and only if $ x_u y = y x_u $ for all $ y \in \mathcal{D}_v $ and $ u,v \in G$,
if and only if $ x_u \in Z(\mathcal{D}) $ for all $ u \in G$.

If $K=T$, the assertion about the support is clear from Equation \eqref{def:beta};
also $Z(\De)=\RR$ unless $\mathbb{C} \cong \De \subseteq Z(\cD)$. 
Now assume that $ K \neq T $ and consider $\mathrm{Int}(X_s)$ for some $ s \in K $.
This is an automorphism that acts trivially on $\De$,
so by Lemma \ref{lem:cocycle},
$\mathrm{Int}(X_s)(X_t) = X_t \nu(t) $ for some $1$-cocycle $\nu$.
In particular, $\beta(s,t)=\nu(t)$, $t\in K$,
takes values in $\mathbb{R}^{\times}$ by Lemma \ref{lem:sses}.
By definition of $K$ and $\beta$, we have $\supp Z(\cD)\subseteq\rad\beta$.
To prove the opposite inclusion, assume that $ s \in \rad\beta $.
Then the corresponding $\nu$ is in the kernel of the restriction map,
so $ \nu(t) = z \in U $ for all $ t \in T \setminus K $.
Pick $ w \in \De \cong \mathbb{C} $ such that $ w \overline{w}^{-1} = z $,
then $ X_s w $ commutes with the homogeneous elements of degree $t$,
so it lies in the center of $\mathcal{D}$.
\end{proof}

\subsection{Involutions}

Now suppose $\varphi$ is an involution on the graded algebra $\mathcal{R}$.
By \cite[Theorem 2.57]{EK13}, $\varphi$ corresponds to a pair $ ( \varphi_0 , B ) $,
where $\varphi_0$ is an antiautomorphism of the graded algebra $\mathcal{D}$,
and $ B : \mathcal{V} \times \mathcal{V} \rightarrow \mathcal{D} $
is a nondegenerate homogeneous $\varphi_0$-sesquilinear form 
(that is, linear over $\cD$ with respect to the second variable and $\varphi_0$-semilinear over $\cD$ with respect to the first variable).
The correspondence is given by
\begin{equation}\label{eq:GradAdj}
B ( rv , w ) = B ( v , \varphi(r) w )
\end{equation}
for all $ v,w \in \mathcal{V} $ and $ r \in \mathcal{R} $.
Note that $ ( \varphi_0 , B ) $ is unique up to the following operation:
$B$ can be replaced by $dB$, where $d$ is a nonzero homogeneous element of $\mathcal{D}$,
and $\varphi_0$ is simultaneously replaced by $ \mathrm{Int}(d) \varphi_0 $.
The pair $ ( \varphi_0^{-1} , \overline{B} ) $,
where $ \overline{B}(v,w) := \varphi_0^{-1} (B(w,v)) $ for all $ v,w \in \mathcal{V} $,
determines $ \varphi^{-1} = \varphi $.
Hence, there exists a nonzero homogeneous element $ \delta \in \mathcal{D} $ such that
$ \overline{B} = \delta B $ and $ \varphi_0^{-1} = \mathrm{Int}(\delta) \varphi_0 $.
Clearly $ \delta \in \De $.

We will assume throughout that $\mathcal{R}$ is central as a graded algebra with involution, that is,
$ \mathrm{Sym} (Z(R)_e,\varphi) := \{ r \in Z(R)_e \mid \varphi(r) = r \} = \mathbb{R} $.

\begin{lemma}\label{lem:Delta}
$B$ can be chosen so that $ \delta \in \{ \pm 1 \} $
(hence $ \varphi_0^2 = \mathrm{id}_{\mathcal{D}} $)
and $\varphi_0$ restricts to the conjugation on $\De$.
\end{lemma}

\begin{proof}
If we replace $B$ by $dB$ and $\varphi_0$ by $ \mathrm{Int}(d) \varphi_0 $,
then $\delta$ is replaced by:
\begin{equation}\label{eq:delta}
\delta' = \delta \varphi_0(d) d^{-1}
\end{equation}
Indeed, we have $ \overline{dB}(v,w) =
\varphi_0^{-1} ( \mathrm{Int}(d^{-1}) (dB(w,v)) ) =
\varphi_0^{-1}(d) \varphi_0^{-1}(B(w,v)) =
\mathrm{Int}(\delta)(\varphi_0(d)) (\delta B(v,w)) =
(\delta \varphi_0(d) d^{-1}) dB(v,w) $.
Applying Equation \eqref{eq:delta} to $ \delta B = \overline{B} $, we get:
$ B = \overline{\overline{B}} =
\delta \varphi_0(\delta) \delta^{-1} \overline{B} =
\delta \varphi_0(\delta) B $,
so $ \delta \varphi_0 (\delta) = 1 $.
Consider the possibilities for $\De$:
\begin{itemize}
\item $ \De = \mathbb{R} $.
We know that $\varphi_0$ is the identity on $\De$,
so $ \delta^2 = 1 $ and $ \delta \in \{ \pm 1 \} $.

\item $ \De \cong \mathbb{H} $.
Replacing $B$ by a suitable $qB$ with $ q \in \De $,
we may assume that $\varphi_0$ restricts to the standard involution (conjugation) on $\De$.
Since $ \delta \in \De $,
$\varphi_0$ is involutive on $C_{\mathcal{D}}(\De)$,
and hence $ \varphi_0^2 = \mathrm{id}_{\mathcal{D}} $.
It follows that $ \delta \in Z(\mathcal{D}) \cap \De = \mathbb{R} $,
and it has to be $+1$ or $-1$.
\item $ \De \cong \mathbb{C} $.
If $\De$ is central in $\mathcal{D}$,
then $\varphi_0$ must be nontrivial on $\De$ by assumption,
because the restrictions of $\varphi$ and $\varphi_0$
to $ Z(\mathcal{R}) = Z(\mathcal{D}) $ coincide by Equation \eqref{eq:GradAdj}.
If $\De$ is not central and $ \varphi_0 \vert_{\De} $ is trivial,
pick a nonzero homogeneous $ d \in \mathcal{D} $ such that
$ \mathrm{Int}(d) \vert_{\De} $ is nontrivial and replace $B$ with $dB$.
Thus, we may assume that $\varphi_0$ is the complex conjugation on $\De$.
So $ 1 = \delta \varphi_0 (\delta) = \vert \delta \vert^2 $;
pick $ z \in \De $ such that $ z \overline{z}^{-1} = \delta $,
then replacing $B$ with $zB$ as in Equation \eqref{eq:delta} yields $ \delta' = 1 $.
\end{itemize}
\end{proof}

\begin{remark}\label{rem:delta}
If $ \De \cong \mathbb{C} $ then 
we can make $\delta$ equal to $+1$ or $-1$ as we wish,
since multiplying $B$ by an imaginary unit $ \mathbf{i} \in \De $ changes $\delta$ to $-\delta$.
Of course, we must simultaneously adjust $\varphi_0$ unless $ \De \subseteq Z(\cD) $.
In this latter case, we will always make the choice $ \delta = 1 $.
\end{remark}

Equation \eqref{eq:delta} suggests the following equivalence relation:

\begin{definition}\label{def:equiv_of_involutions}
Let $\varphi_0$ and $\varphi_0'$ be degree-preserving involutions on a graded division algebra $\cD$.
We will write $\varphi_0\sim\varphi_0'$ if $\varphi_0'=\Int(d)\varphi_0$ for some nonzero homogeneous $d\in\cD$ 
such that $\varphi_0(d)\in\{\pm d\}$.
\end{definition}

If $B$ is $\varphi_0$-sesquilinear and satisfies $\overline{B}\in\{\pm B\}$ then, 
for any $\varphi_0'\sim\varphi_0$, the form $B$ can be replaced by $B'$ that is 
$\varphi_0'$-sesquilinear and also satisfies $\overline{B'}\in\{\pm B'\}$. 

\begin{notation}\label{nota:scaling_Xt}
Fix $B$ as in Lemma \ref{lem:Delta}.
Recall that in the case $ \De \cong \mathbb{H} $,
we take $ X_t \in C_{\mathcal{D}}(\De) $.
Then, as well as in the case $ \De = \mathbb{R} $, we have
$ \varphi_0 (X_t) = \eta(t) X_t $, where $ \eta(t) \in \{ \pm 1 \} $
does not depend on the choice of $X_t$.
In the case $ \De \cong \mathbb{C} $, recall 
$ K := \{ t \in T \mid \mathrm{Int} (X_t) \vert_{\De} =
\mathrm{id}_{\De} \} $.
We have $ \varphi_0 (X_t) = X_t \eta(t) $, $ \eta(t) \in \De $,
and there are two possibilities:
\begin{itemize}
\item
If $ t \in T \setminus K $, then for all $ z \in \De $,
$ \varphi_0 ( X_t z ) = \overline{z} X_t \eta(t) = ( X_t z ) \eta(t) $.
Hence $ X_t = \varphi_0^2 (X_t) = \varphi_0 ( X_t \eta(t) ) = X_t \eta(t)^2 $.
So $ \eta(t) \in \{ \pm 1 \} $,
and $\eta(t)$ does not depend on the choice of $X_t$.
\item
If $ t \in K $, then for all $ z \in \De $, we have
$ \varphi_0 ( X_t z ) = \overline{z} X_t \eta(t) =
X_t \overline{z} \eta(t) = ( X_t z ) \eta(t) z^{-1} \overline{z} $.
Then we can replace $X_t$ by $ X_t z $ so that $ \eta(t) \in \mathbb{R} $
and, moreover, we can control the sign of $\eta(t)$.
In fact $ \eta(t) \in \{ \pm 1 \} $,
because $ X_t = \varphi_0^2 (X_t) = X_t \eta(t)^2 $.
\emph{We choose $X_t$ so that $ \eta(t) = \delta $ when $ t \in K $.}
\end{itemize}
Thus, in all cases, we have 
\[
\varphi_0(X_t)=\eta(t)X_t\quad\text{and}\quad\eta(t)\in\{\pm 1\}\quad\text{for all}\quad t\in T.
\]
\end{notation}

\begin{lemma}\label{lem:eta_is_good}
For all $s,t\in T$, we have $\eta(st^2)=\eta(s)$.
\end{lemma}

\begin{proof}
Consider the case $\De\cong\CC$. First, note that $s\in K$ if and only if $st^2\in K$. By our convention, if $s\in K$ then $\eta(st^2)=\delta=\eta(s)$. 
Suppose $s\notin K$. Then $\eta(st^2)$ does not depend on the choice of a nonzero element in the component $\cD_{st^2}$, 
so we may use the element $X_tX_sX_t$, which gives 
$\varphi_0(X_tX_sX_t)=\varphi_0(X_t)\varphi_0(X_s)\varphi_0(X_t)=\eta(s)\eta(t)^2 X_tX_sX_t=\eta(s)X_tX_sX_t$.

The same argument works in the cases $\De=\RR$ and $\De\cong\HH$.
\end{proof}

\begin{lemma}\label{lem:equiv_of_involutions}
$\varphi_0\sim\varphi_0'$ if and only if $\varphi_0'\varphi_0^{-1}$ is an inner automorphism.
\end{lemma}

\begin{proof}
The ``only if'' part is clear. For the ``if'' part, use Lemma \ref{lem:InnAut} to find a nonzero homogeneous
$d\in\cD$ such that $\varphi_0'=\Int(d)\varphi_0$. Since $\varphi_0\Int(d)=\Int(\varphi_0(d)^{-1})\varphi_0$ and 
both $\varphi_0$ and $\varphi_0'$ are involutions, we obtain $\Int(d)\Int(\varphi_0(d)^{-1})=\mathrm{id}_\cD$,
that is, $\varphi_0(d)=\lambda d$ for some $\lambda\in Z(\cD)$. 
Since we also have $\lambda\in\De$, we conclude that $\lambda\in\RR$ 
except possibly in the case $ \mathbb{C} \cong \De \subseteq Z(\cD) $. In this latter case, we can replace 
$d$ by $dz$ as in Notation \ref{nota:scaling_Xt}, for a suitable $z\in\De$, so that $\lambda\in\RR$, and this 
change does not affect $\Int(d)$. Since $\varphi_0$ is an involution, it follows that $\lambda\in\{\pm 1\}$.
\end{proof}

\subsection{More properties of the graded-division algebra}

The existence of the involution $\varphi_0$ as in Lemma \ref{lem:Delta}
imposes restrictions on the graded division algebra $\cD$.

\begin{lemma}\label{lem:normbeta}
The alternating bicharacter $\beta$ takes values in the unit circle.
Thus, in view of Lemma \ref{lem:suprad}, $\beta$ takes values in $ \{ \pm 1 \} $ except possibly in the case
$ \mathbb{C} \cong \De \subseteq Z(\cD) $.
\end{lemma}

\begin{proof}
Applying $\varphi_0$ to $ X_s X_t = \beta(s,t) X_t X_s $,
we get $ X_t X_s = \overline{\beta(s,t)} X_s X_t $,
so $ \vert \beta(s,t) \vert^2 = \beta(s,t) \overline{\beta(s,t)} = 1 $.
\end{proof}

\begin{lemma}\label{lem:Xt2}
If $ \mathbb{C} \cong \De \nsubseteq Z(\cD) $,
then $X_t^2$ is central for all $ t \in T \setminus K $.
\end{lemma}

\begin{proof}
Let $ s \in K $, then
$ \beta(s,t^2)X_{t^2} = \mathrm{Int}(X_s)(X_{t^2}) = X_{t^2} \nu(t^2) $
as in Lemma \ref{lem:cocycle}.
Because of Lemma \ref{lem:normbeta},
$ \beta(s,t^2) \in \{ \pm 1 \} $.
But $ \nu(t^2) = \vert \nu(t) \vert^2 > 0 $,
therefore $ t^2 \in \mathrm{rad}(\beta) $.
So $X_t^2$ commutes with all $X_s$, $ s \in K $.
Of course, $X_t^2$ commutes with $X_t$,
and with $\De$, so it is in the center.
\end{proof}

\begin{lemma}\label{lem:comm1}
Every automorphism of the graded algebra $\cD$ commutes with $\varphi_0$
except possibly in the case $ \mathbb{C} \cong \De \subseteq Z(\cD) $.
\end{lemma}

\begin{proof}
Recall that $ \varphi_0 (X_t) = \eta(t) X_t $,
where $\eta$ takes values in $ \{ \pm 1 \} $.
Let $\psi_0$ be an automorphism of the graded algebra $\cD$.
First, if $ \psi_0 \vert_{\De} $ is nontrivial,
which can happen only if $ \De \cong \CC $ or $\HH$,
we may compose $\psi_0$ with $\Int(d)$ for a suitable nonzero homogeneous $d$
so that $ \Int(d) \psi_0 \vert_{\De} = \mathrm{id}_{\De} $,
namely, $ d \in \De $ in the case $ \De \cong \HH $,
and $ d= X_t $, $ t \in T \setminus K $, in the case $ \De \cong \CC $.
Now, $ \varphi_0(d) d $ is central:
$ \varphi_0(d) d = \vert d \vert^2 $ in the case $ \De \cong \HH $,
and $ \varphi_0(d) d = \pm X_t^2 $ in the case $ \De \cong \CC $,
so we can use Lemma \ref{lem:Xt2}.
Hence $\varphi_0$ commutes with $\Int(d)$.

Therefore, we may assume that $ \psi_0 \vert_{\De} = \mathrm{id}_{\De} $,
and write $\psi_0$ as in Lemma \ref{lem:cocycle},
$ \psi_0(X_t) = X_t \nu(t) $.
As $\eta$ takes values in $ \{ \pm 1 \} $,
and also $ \nu(t) \in \RR^{\times} $ for $ t \in K $
(see Lemma \ref{lem:sses}),
it is clear that $ \varphi_0 \psi_0 = \psi_0 \varphi_0 $
in the cases $ \De = \mathbb{R} $ and
$ \De \cong \mathbb{H} $ (where $K=T$).
In the case $ \De \cong \mathbb{C} $,
for the same reasons,
the restrictions of $\varphi_0$ and $\psi_0$ to
$ \cD_K := \bigoplus_{ s \in K } \cD_s $ commute with each other.
Finally, consider $ t \in T \setminus K $ in the case $ \De \cong \mathbb{C} $,
then $ \psi_0 \varphi_0 (X_t) = X_t \nu(t) \eta(t) $,
and $ \varphi_0 \psi_0 (X_t) = \varphi_0 ( X_t \nu(t) ) =
\overline{\nu(t)} X_t \eta(t) = X_t \nu(t) \eta(t) $.
\end{proof}

\begin{lemma}\label{lem:comm2}
Suppose that $ \mathbb{C} \cong \De \subseteq Z(\cD) $, that is, $\cD$ can be given a complex structure (in two ways) 
that makes it a graded $\mathbb{C}$-algebra. Then the automorphisms of the graded algebra $\cD$ are as follows:
\begin{itemize}
\item
The $\CC$-linear automorphisms $\psi_0$ (that is, those whose restriction to $\De$ is the identity) 
are given by $ \psi_0(X_t) = X_t \nu(t) $, where $ \nu \in \mathrm{Hom}(T,\mathbb{C}^{\times}) $.
\item
The $\CC$-antilinear automorphisms $\psi_0$ (that is, those whose restriction to $\De$ is the conjugation)
exist only if $\beta$ takes values in $ \{ \pm 1 \} $. They are given by $ \psi_0(X_t) = X_t \nu(t) $,
where the map $ \nu : T \to \mathbb{C}^{\times} $ satisfies
$ \nu(st) = \nu(s) \nu(t) \beta(s,t) $ for all $ s,t \in T $.
\end{itemize}
In both cases, $\psi_0$ commutes with $\varphi_0$ if and only if $\nu$ takes values in $\RR^\times$.
\end{lemma}

\begin{proof}
If $\psi_0$ is $\CC$-linear then
we already know from Lemma \ref{lem:cocycle} that it has the form $ \psi_0(X_t) = X_t \nu(t) $,
where $ \nu \in Z^1 (T, \De^{\times} ) = \mathrm{Hom} (T, \mathbb{C}^{\times} ) $.

Assume that $\psi_0$ is $\CC$-antilinear. Then $ \psi_0 = \psi'_0 \varphi_0 $,
where $\psi'_0$ is a $\mathbb{C}$-linear antiautomorphism of the graded algebra $\cD$.
Clearly, $\psi'_0$ is given by $ \psi'_0 (X_t) = X_t \nu(t) $, 
where $ \nu(t) \in \De^{\times} = \CC^{\times} $.
Applying $\psi'_0$ to Equation \eqref{def:beta}, we get
$ \psi'_0(X_t) \psi'_0(X_s) = \beta(s,t) \psi'_0(X_s) \psi'_0(X_t) =
\beta(s,t)^2 \psi'_0(X_t) \psi'_0(X_s) $,
so $\beta$ takes values in $ \{ \pm 1 \} $.
Recall that $\varphi_0$ restricts to the conjugation on $\De$ and 
$ \varphi_0(X_t) = X_t $ for all $ t \in T $ (since $ T = K $ and we may assume 
$ \delta = 1 $ by Remark \ref{rem:delta}).
Hence, we have $\psi_0(X_t)=X_t \nu(t)$.

Now let $\psi'_0$ be a $\CC$-linear map given by $ \psi'_0 (X_t) = X_t \nu(t) $, where $ \nu(t) \in \CC^{\times} $. 
Then, on the one hand, $\psi_0'(X_s X_t)=X_s X_t \nu(st)$ and, on the other hand, 
$\psi'_0(X_t) \psi'_0(X_s) = \beta(s,t) X_s X_t \nu(s)\nu(t)$ by Equation \eqref{def:beta}.
Therefore, $\psi'_0$ is an antiautomorphism of $\cD$ if and only if
$\nu(st)=\nu(s)\nu(t)\beta(s,t)$ for all $ s,t \in T $. 

Finally, whether $\psi_0$ is $\CC$-linear or $\CC$-antilinear, 
we have $\psi_0\varphi_0(X_t) = X_t\nu(t)$ and $\varphi_0\psi_0(X_t)=\overline{\nu(t)}X_t$.
Therefore, $ \varphi_0 \psi_0 = \psi_0 \varphi_0 $
if and only if $ \nu(t) = \overline{\nu(t)} $ for all $ t \in T $.
\end{proof}

\subsection{The structure theorem}

Let $ g_0 \in G $ be the degree of $B$.
Define $ (G/T)_{g_0} := \{ gT \mid g_0 g^2 \in T \} $
and $ \tau : (G/T)_{g_0} \rightarrow T $ by $ \tau(gT) = g_0 g^2 $
where $g$ is in the fixed transversal for $T$ in $G$.
In other words, $\tau(x)=g_0\xi(x)^2$ whenever $g_0\xi(x)^2\in T$.

\begin{remark}\label{rem:tau}
If $T$ is an elementary $2$-group,
then the definition of $\tau$ does not depend on the transversal.
\end{remark}

Recall that the isomorphism class of $\cV$ is determined by a function
$ \kappa : G/T \rightarrow \mathbb{Z}_{ \geq 0 } $
(a multiset) with a finite support $\{x_1,\ldots,x_s\}$.
We will now see that the existence of $B$ imposes restrictions on $\kappa$,
and also determines another function,
$ \TrSig : G/T \rightarrow \mathbb{Z} $.

\begin{definition}\label{def:MultFunct}
We will say that a map $ \kappa : G/T \rightarrow \mathbb{Z}_{ \geq 0 } $
is an \emph{admissible multiplicity function} if it satisfies the following conditions:
\begin{enumerate}
\item[(a)] the support of $\kappa$ is finite;
\item[(b)] $ \kappa ( g_0^{-1} x^{-1} ) = \kappa(x) $ for all $ x \in G/T $;
\item[(c)] for $ \De = \mathbb{R} $:\\
$ \kappa(x) \equiv 0 \pmod 2 $ if $ x \in (G/T)_{g_0} $
and $ \eta(\tau(x)) = -\delta $;
\item[(c$'$)] for $ \De \cong \mathbb{C} $:\\
$ \kappa(x) \equiv 0 \pmod 2 $ if $ x \in (G/T)_{g_0} $,
$ \tau(x) \in T \setminus K $
and $ \eta(\tau(x)) = -\delta $.
\end{enumerate}
The set of all admissible multiplicity functions will be denoted by
$\MulSet(G,\cD,\varphi_0,g_0,\delta)$.
(Thanks to Lemma \ref{lem:eta_is_good}, it does not depend on our choice of transversal.)
As before, we will write $k_i=\kappa(x_i)$ and $ \vert \kappa \vert := \sum_{ x \in G/T } \kappa(x) = k_1+\cdots+k_s$.
\end{definition}

\begin{definition}\label{def:SignFunct}
For a given admissible multiplicity function $\kappa$,
we will say that a map $ \TrSig : G/T \rightarrow \mathbb{Z} $
is a \emph{signature function} if it satisfies the following conditions:
\begin{enumerate}
\item[(i)] $ \vert \TrSig(x) \vert \leq \kappa(x) $ for all $ x \in G/T $;
\item[(ii)] for $ \De = \mathbb{R} $ and $ \De \cong \mathbb{H} $:\\
$ \TrSig (x) = 0 $ unless $ x \in (G/T)_{g_0} $ and
$ \eta(\tau(x)) = \delta $;\\
$ \TrSig(x) \equiv \kappa(x) \pmod 2 $ if $ x \in (G/T)_{g_0} $ and
$ \eta(\tau(x)) = \delta $;
\item[(ii$'$)] for $ \De \cong \mathbb{C} $:\\
$ \TrSig (x) = 0 $ unless $ x \in (G/T)_{g_0} $ and
$ \tau(x) \in K $;\\
$ \TrSig(x) \equiv \kappa(x) \pmod 2 $ if $ x \in (G/T)_{g_0} $ and
$ \tau(x) \in K $.
\end{enumerate}
The set of all signature functions for a given $\kappa$ will be denoted by
$ \SigSet ( G , \allowbreak \cD , \allowbreak \varphi_0 ,
\allowbreak g_0 , \allowbreak \kappa , \allowbreak \delta ) $.
(It does not depend on our choice of transversal.) 
\end{definition}

Recall the isotypic components $\cV_i=\cV_{g_i}\otimes_{\De} \cD$, where $g_i=\xi(x_i)$, 
and observe that $ B ( \mathcal{V}_i , \mathcal{V}_j ) = 0 $ unless $ g_0 g_i g_j \in T $.
Hence $B$ pairs $\mathcal{V}_i$ with itself if $ x_i = g_i T \in (G/T)_{g_0} $
and with $\mathcal{V}_j$, $ j \neq i $, otherwise.
This shows that $\kappa$ satisfies Property (b) of Definition \ref{def:MultFunct}.
For $i$ such that $ x_i \in (G/T)_{g_0} $,
we get a nondegenerate homogeneous $\varphi_0$-ses\-qui\-lin\-ear form
$ B \vert_{ \mathcal{V}_i \times \mathcal{V}_i } $
whose values on $ \mathcal{V}_{g_i} \times \mathcal{V}_{g_i} $
lie on $\mathcal{D}_{t_i}$, where $ t_i := \tau(x_i) = g_0g_i^2$.
Define $ B_i : \mathcal{V}_{g_i} \times \mathcal{V}_{g_i} \rightarrow \De $
by
\begin{equation}\label{eq:Bi}
B(v,w) = X_{t_i} B_i(v,w) \text{ for all } v,w \in \mathcal{V}_{g_i} \text{.}
\end{equation}
Then $B_i$ is a nondegenerate form on the $\De$-vector space $\cV_{g_i}$ and it is
$ \mathrm{Int} (X_{t_i}^{-1}) \varphi_0 \vert_{\De} $-sesquilinear.
Accordingly, we consider
$ \overline{B_i} (v,w) = \varphi_0 ( X_{t_i} B_i(w,v) X_{t_i}^{-1} ) $
for all $ v,w \in \mathcal{V}_{g_i} $.
Using Equation \eqref{eq:delta}, we get
$ \overline{B_i} = \delta \varphi_0 (X_{t_i}^{-1}) X_{t_i} B_i = \delta \eta(t_i) B_i $.
Depending on the type of $\De$ we get the following:
\begin{itemize}
\item If $ \De = \mathbb{R} $, then $ B_i(w,v) = \delta \eta(t_i) B_i(v,w) $,
so $B_i$ is either symmetric or skew-symmetric.
\item If $ \De \cong \mathbb{H} $,
then $ \varphi_0 (B_i(w,v)) = \delta \eta(t_i) B_i(v,w) $,
so $B_i$ is either hermitian or skew-hermitian.
\item If $ \De \cong \mathbb{C} $ and $ t_i \in T \setminus K $,
then $ B_i(w,v) = \delta \eta(t_i) B_i(v,w) $,
so $B_i$ is either symmetric or skew-symmetric.
\item If $ \De \cong \mathbb{C} $ and $ t_i \in K $,
then $ \varphi_0 (B_i(w,v)) = \delta \eta(t_i) B_i(v,w) = B_i(v,w) $,
so $B_i$ is hermitian.
\end{itemize}
Since nondegenerate skew-symmetric forms
(over $\mathbb{R}$ or $\mathbb{C}$)
exist only in even dimension,
$\kappa$ satisfies conditions (c)-(c$'$) of Definition \ref{def:MultFunct}.
Also, the form $B_i$ is hermitian (over $\RR$, $\CC$ or $\HH$) and therefore has inertia $(p_i,q_i)$ in the following cases:
\begin{itemize}
\item for $ \De = \mathbb{R} $ or $ \De \cong \mathbb{H} $,
when $ \eta(t_i) = \delta $;
\item for $ \De \cong \mathbb{C} $, when $ t_i \in K $.
\end{itemize}
So we can define a map $ \TrSig : G/T \rightarrow \mathbb{Z} $
as $ \TrSig (x_i) = p_i - q_i $ if $B_i$ has inertia,
and $ \TrSig (x) = 0 $ for all other $ x \in G/T $.
Since $ p_i + q_i = k_i = \kappa(x_i) $,
this map $\TrSig$ satisfies the conditions of Definition \ref{def:SignFunct},
that is, it is a signature function.

Recall that, if we choose a homogeneous $\mathcal{D}$-basis of $\mathcal{V}$,
then the elements of $ \mathcal{R} = \mathrm{End}_{\mathcal{D}} (\mathcal{V}) $
can be identified with matrices in $M_k(\mathcal{D})$, where 
$ k = \vert \kappa \vert$.
Also, Equation \eqref{eq:GradAdj} becomes
\begin{equation}\label{eq:MatrInv}
\varphi(X) = \Phi^{-1} \varphi_0(X^T) \Phi
\end{equation}
for all $ X \in M_k(\mathcal{D}) $,
where $\Phi$ is the matrix in $M_k(\mathcal{D})$
representing $B$ with respect to the chosen basis,
and $\varphi_0$ acts entrywise.

We can relabel $ ( g_1 , \ldots , g_s ) $ so that
the first $m$ entries satisfy $ g_0 g_i^2 \in T $ and
\[ g_{m+1} g_{m+2} \equiv \ldots \equiv g_{m+2r-1} g_{m+2r} \equiv g_0^{-1} \pmod{T} \]
where $ m + 2r = s $.
Write, as in Equation \eqref{eq:Bi}, $ B(v,w) = X_{t_{m+j}} B_{m+j}(v,w) $,
for all $ v \in \mathcal{V}_{g_{m+2j-1}} $ and $ w \in \mathcal{V}_{g_{m+2j}} $,
where $ t_{m+j} := g_0 g_{m+2j-1} g_{m+2j} $.
We can repeat the same argument to get
$ \overline{B_{m+j}} = \delta \eta(t_{m+j}) B_{m+j} $.
Choosing $\De$-bases in $ \mathcal{V}_{g_1} , \ldots , \mathcal{V}_{g_m} $
(which are then $\mathcal{D}$-bases in $ \mathcal{V}_1 , \ldots , \mathcal{V}_m $)
to bring $B_i$ to canonical form,
and choosing $\De$-bases in $\mathcal{V}_{g_{m+2j-1}}$ and $\mathcal{V}_{g_{m+2j}}$
that are dual with respect to $B_{m+j}$, we obtain:
\begin{equation}\label{eq:Phi}
\Phi = X_{t_1} S_1 \oplus \ldots \oplus X_{t_m} S_m
\oplus X_{t_{m+1}} S_{m+1} \oplus \ldots \oplus X_{t_{m+r}} S_{m+r}
\end{equation}
where the matrices $S_i$ are as follows.
For $ i = m+j $, $ j = 1 , \ldots , r $, we have $ k_{m+2j-1} = k_{m+2j} $ and
\[ S_i = \begin{pmatrix}
0 & I_{k_{m+2j}} \\
\delta \eta(t_{m+j}) I_{k_{m+2j}} & 0
\end{pmatrix}, \]
whereas for $ i = 1 , \ldots , m $ we have the following cases:
\begin{itemize}
\item For $ \De = \mathbb{R} $ or $ \De \cong \mathbb{H} $
with $ \eta(t_i) = \delta $,
and for $ \De \cong \mathbb{C} $ with $ t_i \in K $:
$ S_i = I_{p_i,q_i} $ (the diagonal matrix with $\pm 1$ on the main diagonal: the first $p_i$ entries are $1$ 
and the remaining $q_i$ entries are $-1$).
\item For $ \De = \mathbb{R} $ with $ \eta(t_i) = -\delta $,
and for $ \De \cong \mathbb{C} $ with $ t_i \in T \setminus K $ and $ \eta(t_i) = -\delta $:
\[ S_i = \begin{pmatrix} 0 & I_{k_{i/2}} \\ -I_{k_{i/2}} & 0 \end{pmatrix}. \]
\item For $ \De \cong \mathbb{C} $ with $ t_i \in T \setminus K $
and $ \eta(t_i) = \delta $: $ S_i = I_{k_i} $.
\item For $ \De \cong \mathbb{H} $ with $ \eta(t_i) = -\delta $:
$ S_i = \mathbf{i} I_{k_i} $  ($\mathbf{i}$ is a fixed imaginary unit in $\De$).
\end{itemize}

To summarize:

\begin{theorem}\label{th:struct}
Let $\mathcal{R}$ be a $G$-graded real associative algebra
that is graded-simple and satisfies the descending chain condition on graded left ideals,
and such that the identity component $\cR_e$ has finite dimension.
Let $\varphi$ be an involution on the graded algebra $\mathcal{R}$
such that $\mathcal{R}$ is central as a graded algebra with involution.
Then there exists a real graded division algebra $\cD$ with $\De$ isomorphic to $\RR$, $\CC$ or $\HH$
and a degree-preserving involution $\varphi_0$ on $\cD$ restricting to the conjugation on $\De$
such that, as a graded algebra with involution, $\cR$ is isomorphic to $M_k(\cD)$,
where the grading is given by Equation \eqref{eq:MatrGrad},
and the involution is given by Equations \eqref{eq:MatrInv} and \eqref{eq:Phi}.
\qed
\end{theorem}

\begin{remark}
It may be convenient to choose the transversal for $T$ in $G$ so that,
for all $ x' \neq x'' $ in $G/T$ satisfying $ g_0 x' x'' = T $,
we have $ g_0 \xi(x') \xi(x'') = e $.
Then, in Equation \eqref{eq:Phi},
each of the $ t_{m+1} , \ldots , t_{m+r} $ is just $e$.
However, this choice of transversal depends on $g_0$.
\end{remark}

Suppose we are given $(\cD,\varphi_0)$ as in Theorem \ref{th:struct} and $\delta\in\{\pm 1\}$. 
Fix a transversal for $T$ in $G$ and, for each $t\in T$, an element $X_t\in\cD_t$ 
as described in Notation~\ref{nota:scaling_Xt}. Then the data that determine $(\cR,\varphi)$ are 
$(g_0,\kappa,\TrSig)$.
Conversely, let $ g_0 \in G $,
$ \kappa \in \MulSet(G,\cD,\varphi_0,g_0,\delta) $ and
$ \TrSig \in \SigSet(G,\cD,\varphi_0,g_0,\kappa,\delta) $.
Then we know that, by means of Equation \eqref{eq:MatrGrad},
$\kappa$ makes $ \cR := M_k(\cD) $, with $ k =  \vert \kappa \vert  $,
a graded-simple algebra that satisfies
the descending chain condition on graded left ideals
and such that $\cR_e$ has finite dimension.
We can construct $ \Phi \in M_k(\cD) $ as in Equation \eqref{eq:Phi}, 
with $p_i=\frac12(\kappa(x_i)+\TrSig(x_i))$ and $q_i=\frac12(\kappa(x_i)-\TrSig(x_i))$ where appropriate,
and define $ \varphi : M_k(\cD) \rightarrow M_k(\cD) $ by Equation \eqref{eq:MatrInv},
which is equivalent to Equation \eqref{eq:GradAdj} for the $\varphi_0$-sesquilinear form $B$ represented by $\Phi$.
Since $ \varphi_0 ( \Phi^T ) = \delta \Phi $
(equivalently, $ \overline{B} = \delta B $),
$\varphi$ is an involution.
Since $B$ is homogeneous, $ \varphi (\cR_g) = \cR_g $.
Finally, $M_k(\cD)$ is central as a graded algebra with involution
because $ Z(\cR)_e = Z(\cD)_e $ and $\varphi_0$ restricts to the conjugation on $\De$.

\begin{definition}\label{def:Mdata}
We will denote the graded algebra with involution $(M_k(\cD),\varphi)$ constructed above by 
$M(\cD,\varphi_0,g_0,\kappa,\TrSig,\delta)$.
\end{definition}

\section{Classification of associative algebras with involution up~to~isomorphism}\label{se:assoc_classification}

Let $(\cR,\varphi)$ and $(\cR',\varphi')$ be graded algebras with involution as in Theorem \ref{th:struct},
which are isomorphic to $M(\cD,\varphi_0,g_0,\kappa,\TrSig,\delta)$
and $M(\cD',\varphi_0',g_0',\kappa',\TrSig',\delta')$, respectively, as in Definition \ref{def:Mdata}.
The purpose of this section is to determine conditions under which these two objects are isomorphic to each other. 

\subsection{Preliminary remarks}

Recall that, by \cite[Theorem 2.10]{EK13},
an isomorphism $ \psi : \cR \rightarrow \cR' $
between the $G$-graded algebras $ \cR = \End_{\cD}(\cV) $ and $ \cR' = \End_{\cD'}(\cV') $
(for now, without regard to involutions) corresponds to a pair $(\psi_0,\psi_1)$ where
$ \psi_0 : \cD \rightarrow \cD' $ is an isomorphism of $G$-graded algebras and 
$ \psi_1 : \cV^{[g]} \rightarrow \cV' $, for some $g\in G$,
is an isomorphism of $G$-graded spaces (in other words, an invertible linear map $\cV\to\cV'$ of degree $g$) 
such that $ \psi_1(vd) = \psi_1(v) \psi_0(d) $ for all $ v \in \cV $ and $ d \in \cD $. 
The correspondence is given by $ \psi_1(rv) = \psi(r) \psi_1(v) $
for all $ r \in \cR $ and $v\in\cV$. 
Moreover, $(\psi_0,\psi_1)$ is unique up to replacing
$\psi_0$ by $ \psi_0' = \Int(d^{-1}) \psi_0 $,
where $d$ is a nonzero homogeneous element in $\cD'$,
and simultaneously replacing $\psi_1$ by $\psi_1'$,
where $ \psi_1'(v) = \psi_1(v) d $ for all $ v \in \cV $.

Since it is necessary that $ \cD \cong \cD' $,
we will suppose that $ \cD = \cD' $ and thus $\psi_0$ is
an automorphism of the graded algebra $\cD$, while $\psi_1$ is $\psi_0$-semilinear as a map between $\cD$-modules.
We will denote the group of automorphisms of the graded algebra $\cD$ by $\mathrm{Aut}^G(\cD)$. 
Also, we will write $\Dgr^\times$ for the group of nonzero homogeneous elements of $\cD$.
Note that if $\psi_0$ is inner then, in view of Lemma \ref{lem:InnAut}, 
we have $\psi_0=\Int(d)$ for some $d\in\Dgr^\times$, so
we can adjust the pair $(\psi_0,\psi_1)$ so that $\psi_0=\mathrm{id}_\cD$ and still 
obtain the same isomorphism $\psi$.

In the case $ \CC \cong \De \subseteq Z(\cD) $ (in other words, $\De\cong\CC$ and $K=T$), we will restrict ourselves to 
the following situation: we will assume that $ \De = Z(\cD) $ and $T$ is finite. 
This simplifies the arguments and will be sufficient for our applications to classical Lie algebras 
(but see Subsection \ref{sse:groupoids} for the general situation).
Since here $\cD$ is a twisted group algebra of $T$, we see by a generalization of Maschke's Theorem 
(e.g. \cite[Corollary 10.2.5]{Kar}) that  
our assumption is tantamount to $\cD$ (equivalently, $\cR$) being a finite-dimensional central simple algebra over $\CC$
(which is graded as a $\CC$-algebra). 
We will refer to this as the \emph{$\CC$-central case}. 
Note that in this case the alternating bicharacter $\beta$ must be nondegenerate (see Lemma \ref{lem:suprad}).

The next step is to take into account the involutions. By \cite[Lemma 3.32]{EK13},
if we want $\psi$ to be an isomorphism of graded algebras with involution,
that is, $ \varphi' = \psi \varphi \psi^{-1} $,
it is necessary and sufficient that there exist $ d_0 \in \Dgr^\times $ such that
\begin{equation}\label{eq:relation_B} 
B'( \psi_1(v) , \psi_1(w) ) = \psi_0( d_0 B(v,w) ) 
\end{equation} 
for all $ v,w \in \cV $.
Automatically, $ \Int(d_0) \varphi_0 = \psi_0^{-1} \varphi_0' \psi_0 $.
Note that, since both $\varphi_0$ and $\varphi_0'$ restrict to the conjugation on $\De$, we have $d_0\in C_{\cD}(\De)$ 
and hence $\deg d_0\in K$.
Also, it follows from Equations \eqref{eq:delta} and \eqref{eq:relation_B} that
$ \delta' = \delta \varphi_0(d_0) d_0^{-1} $
and, in particular,  $\varphi_0(d_0)\in\{\pm d_0\}$.
Indeed, consider the form $B'':\cV\times\cV\to\cD$ given by $B''(v,w)=\psi_0^{-1}(B'(\psi_1(v),\psi_1(w)))$. 
One checks that $B''$ is $(\psi_0^{-1}\varphi_0'\psi_0)$-sesquilinear and satisfies $\overline{B''}=\delta'B''$.
Now Equation \eqref{eq:relation_B} reads $B''=d_0B$, so we may apply Equation \eqref{eq:delta} (with $d_0$ in place of $d$). 

If we are not in the $\CC$-central case, then $\psi_0$ commutes with $\varphi_0$ and $\varphi_0'$ by Lemma \ref{lem:comm1}, 
so the relation between $\varphi_0$ and $\varphi_0'$ simplifies: 
$\Int(d_0)\varphi_0=\varphi_0'$.
Moreover, adjusting the pair $(\psi_0,\psi_1)$,
we may assume without loss of generality that $ \psi_0 \vert_{\De} = \mathrm{id}_{\De} $
(see the proof of Lemma \ref{lem:comm1}).
Then, by Lemma \ref{lem:cocycle}, we have $ \psi_0(X_t) = X_t \nu(t) $,
where $ \nu \in Z^1 (T, Z(\De)^{\times} ) $.
We define for future reference:
\begin{equation}\label{def:groupA}
A := \{\psi_0\in\mathrm{Aut}^G(\cD) \mid \psi_0 \vert_{\De} = \mathrm{id}_{\De}\}
\cong Z^1 (T, Z(\De)^{\times} ).
\end{equation}

If we are in the $\CC$-central case, then $\psi_0$ does not always commute with the involutions --- see Lemma \ref{lem:comm2}.
However, if $\psi_0$ is $\CC$-linear then, 
by Skolem--Noether Theorem, it is inner and hence can be eliminated.
The existence of $\CC$-antilinear $\psi_0$ forces $\beta$ to take values in $\{\pm 1\}$. 
Since $\beta$ is nondegenerate, this implies that $T$ is an elementary $2$-group (compare with \cite[Lemma 2.50]{EK13}). 
Conversely, if $T$ is an elementary $2$-group then, by \cite[Proposition 2.51]{EK13}, there exists $\nu:T\to\{\pm 1\}$ satisfying 
$\nu(st)=\nu(s)\nu(t)\beta(s,t)$, so we obtain a $\CC$-antilinear automorphism $\psi_0$ sending $X_t\mapsto X_t\nu(t)$,
which generates the group $\mathrm{Aut}^G(\cD)$ modulo the inner automorphisms. Fix such $\nu$ and the corresponding $\psi_0$.
(Note that, if we regard $T$ as a vector space over the field of two elements, then $\nu$ is a quadratic form on $T$ 
with polar form $\beta$.) 
Since $\nu(t)\in\RR^\times$ for all $t\in T$, $\psi_0$ commutes with $\varphi_0$. 
It also commutes with $\Int(d)$ for all $d\in\Dgr^\times$, 
so in this case, too, the relation between $\varphi_0$ and $\varphi_0'$ simplifies: 
$\Int(d_0)\varphi_0=\varphi_0'$. We define for future reference:
\begin{equation}\label{def:groupA_}
A := \begin{cases}
\langle\psi_0\rangle\cong\ZZ_2 & \text{if $T$ is an elementary $2$-group;}\\
1 & \text{otherwise.}
\end{cases}
\end{equation}

We have seen that in all cases $ \Int(d_0) \varphi_0 = \varphi_0' $. 
Thus, for $(\cR,\varphi)$ and $(\cR',\varphi')$ to be isomorphic, 
it is necessary that $\varphi_0\sim\varphi_0'$ in the sense of Definition \ref{def:equiv_of_involutions}. 
Fixing a representative for each equivalence class, we may suppose that $\varphi_0 = \varphi_0'$,
so both $B$ and $B'$ are $\varphi_0$-sesquilinear forms, and they are related by  
Equation \eqref{eq:relation_B} where $d_0\in Z(\cD)$. 

There are two possibilities:
\begin{itemize}
\item If $ \varphi_0 \vert_{Z(\cD)} = \mathrm{id}_{Z(\cD)} $,
then necessarily $ \varphi_0(d_0) = d_0 $ and hence $ \delta' = \delta $
whenever $(\varphi_0,B)$ and $(\varphi_0,B')$
yield isomorphic $(\cR,\varphi)$ and $(\cR',\varphi')$.
Thus, $ \delta \in \{ \pm 1 \} $ is an invariant in this case,
and it suffices to consider $\varphi_0$-sesquilinear forms with a fixed $\delta$.
\item If $ \varphi_0 \vert_{Z(\cD)} \neq \mathrm{id}_{Z(\cD)} $,
then there exists a nonzero homogeneous $ d_0 \in Z(\cD) $
such that $ \varphi_0(d_0) = - d_0 $,
and hence we may adjust $(\varphi_0,B)$
to make $\delta=1$ or $\delta=-1$ as we wish.
\emph{We choose $ \delta = 1 $ and thus
consider only hermitian $\varphi_0$-sesquilinear forms over $\cD$ in this case.}
\end{itemize}
In either case, we will have $\varphi_0(d_0)=d_0$. We define for future reference:
\begin{equation}\label{def:groupC}
C := \{c\in \Dgr^\times \mid c\in Z(\cD)\text{ and }\varphi_0(c)=c\}.
\end{equation}

To summarize: we may suppose without loss of generality that
$ \cD = \cD' $, $ \varphi_0 = \varphi_0' $ and $ \delta = \delta' $.
Under this assumption, we will obtain conditions on 
$(g_0,\kappa,\TrSig)$ and $(g_0',\kappa',\TrSig')$ that are necessary and sufficient for 
$(\cR,\varphi)$ and $(\cR',\varphi')$ to be isomorphic.

\subsection{Extended signature functions}

Recall that the signature function $ \TrSig \in \SigSet(G,\cD,\varphi_0,g_0,\kappa,\delta) $ corresponding to 
a nondegenerate $\varphi_0$-sesquilinear form $B:\cV\times\cV\to\cD$, with $\deg B=g_0$ and $\overline{B}=\delta B$, 
assigns to each $x\in G/T$ the signature of the form 
$X_{\tau(x)}^{-1}B:\cV_{\xi(x)}\times\cV_{\xi(x)}\to\De$ if this signature is defined 
(that is, if $X_{\tau(x)}^{-1}B\vert_{\cV_{\xi(x)}\times\cV_{\xi(x)}}$ is a nondegenerate hermitian form over $\De$)
and zero otherwise. Here $\xi:G/T\to G$ is the section given by our fixed transversal and $\tau(x)=g_0\xi(x)^2$.
It is convenient to define $ \InSig : G \rightarrow \mathbb{Z} $ by 
\begin{equation}\label{eq:DefInSig}
\InSig(h) := 
\begin{cases}
\mathrm{signature}
( X_{ g_0 h^2 }^{-1} B \vert_{ \cV_h \times \cV_h } ) & 
\text{if $ g_0 h^2 \in K $ and $ \eta( g_0 h^2 ) = \delta $;}\\
0 & \text{otherwise.}
\end{cases}
\end{equation}
(Recall that if $ \De = \mathbb{R} $ or $ \De \cong \mathbb{H} $ then $K=T$, and if $ \De \cong \mathbb{C} $ then $\eta(t)=\delta$ for all $t\in K$.)
The relation between $ \TrSig : G/T \rightarrow \mathbb{Z} $
and $ \InSig : G \rightarrow \mathbb{Z} $
is simply $ \TrSig = \InSig \xi $.

\begin{proposition}\label{condition0}
For all $ h \in G $ such that $ t := g_0 h^2 \in K $ and for all $ u \in T $, 
\begin{equation}\label{eq:PropInSig}
\InSig(hu) = \eta(u) \mathrm{sign}
\big( X_{tu^2}^{-1} X_u X_t X_u \big) \InSig(h)
\end{equation}
\end{proposition}

\begin{proof}
First, observe that $ X_{tu^2}^{-1} X_u X_t X_u \in \RR^{\times} $,
so the sign above makes sense.
Indeed, this element belongs to $\De\cap C_{\cD}(\De)=Z(\De)$,
so the result is clear in the cases $ \De = \RR $ and $ \De \cong \HH $.
In the case $ \De \cong \CC $, it suffices to show that
$ \varphi_0 ( X_{tu^2}^{-1} X_u X_t X_u ) = X_{tu^2}^{-1} X_u X_t X_u $.
We compute:
\begin{align*}
\varphi_0 ( X_{tu^2}^{-1} X_u X_t X_u )
& = \varphi_0(X_u) \varphi_0(X_t) \varphi_0(X_u) \varphi_0(X_{tu^2})^{-1}\\
& = \eta(u) \eta(t) \eta(u) \eta(tu^2)^{-1} ( X_u X_t X_u ) X_{tu^2}^{-1}\\
& = X_{tu^2}^{-1} ( X_u X_t X_u ),
\end{align*}
where we have used that $\eta$ takes values $ \pm 1 $
and $ \eta(t) = \eta(tu^2) = \delta $
because $t$ and $tu^2$ are in $K$;
we have also used that $ X_u X_t X_u $ and $X_{tu^2}$ commute
because they are in the same component $\cD_{tu^2}$.

Now, we have to compare $\De$-valued forms
$ X_t^{-1} B \vert_{ \cV_h \times \cV_h } $ and
$ X_{tu^2}^{-1} B \vert_{ \cV_{hu} \times \cV_{hu} } $.
First of all, they either both have signature or both do not.
In the case $ \De \cong \CC $,
$t$ and $tu^2$ are in $K$.
In the cases $ \De = \RR $ and $ \De \cong \HH $,
we have to check that $ \eta(t) = \eta(tu^2) $.
Indeed, $ X_t X_u^2 = \lambda X_{tu^2} $ for some $ \lambda \in \RR $,
hence $ \varphi_0(X_u)^2 \varphi_0(X_t) = \lambda \varphi_0(X_{tu^2}) $,
so $ \eta(t) X_u^2 X_t = \lambda \eta(tu^2) X_{tu^2} =
\eta(tu^2) X_t X_u^2 = \eta(tu^2) X_u^2 X_t $,
where we have used the fact that $\beta$ takes values $ \pm 1 $
(Lemma \ref{lem:normbeta}).

Finally, suppose that the signatures are defined.
We compute, for all $ v,w \in \cV_h $:
\begin{align*}
X_{tu^2}^{-1} B(vX_u,wX_u)
& = X_{tu^2}^{-1} \varphi_0(X_u) B(v,w) X_u\\
& = \eta(u) X_{tu^2}^{-1} X_u X_t ( X_t^{-1} B(v,w) ) X_u;
\end{align*}
the last expression equals
$ \eta(u) X_{tu^2}^{-1} X_u X_t X_u \overline{( X_t^{-1} B(v,w) )} $
if $ \De \cong \CC $ and $ u \in T \setminus K $
and $ \eta(u) X_{tu^2}^{-1} X_u X_t X_u ( X_t^{-1} B(v,w) ) $
otherwise.
The result follows because the map $ v \mapsto vX_u $
is a $\De$-linear isomorphism from $\cV_h$ to $\cV_{hu}$,
respectively from $\overline{\cV_h}$ to $\cV_{hu}$
if $ \De \cong \CC $ and $ u \in T \setminus K $.
(Here we use the bar to denote the conjugate complex vector space: it has the same addition of vectors, 
but the scalar multiplication is twisted by complex conjugation.)
\end{proof}

\begin{definition}
For a given admissible multiplicity function $\kappa$,
we will say that a map $ \InSig : G \rightarrow \mathbb{Z} $
is an \emph{extended signature function} if it satisfies
Equation \eqref{eq:PropInSig} and the following conditions:
\begin{enumerate}
\item[(i)] $ \vert \InSig(h) \vert \leq \kappa(hT) $ for all $ h \in G $;
\item[(ii)] for $ \De = \mathbb{R} $ and $ \De \cong \mathbb{H} $:\\
$ \InSig (h) = 0 $ unless $ g_0 h^2 \in T $ and
$ \eta( g_0 h^2 ) = \delta $;\\
$ \InSig(h) \equiv \kappa(hT) \pmod 2 $ if $ g_0 h^2 \in T $ and
$ \eta( g_0 h^2 ) = \delta $;
\item[(ii$'$)] for $ \De \cong \mathbb{C} $:\\
$ \InSig (h) = 0 $ unless $ g_0 h^2 \in K $;\\
$ \InSig(h) \equiv \kappa(hT) \pmod 2 $ if $ g_0 h^2 \in K $.
\end{enumerate}
The set of all extended signature functions will be denoted by
$ \InSigSet ( G , \allowbreak \cD , \allowbreak \varphi_0 ,
\allowbreak g_0 , \allowbreak \kappa , \allowbreak \delta ) $.
\end{definition}

From Definition \ref{def:SignFunct} and Proposition \ref{condition0}, it is clear that $\InSig$ determined by $B$ 
via Equation \eqref{eq:DefInSig} is an extended signature function.
Conversely, any $ \InSig\in\InSigSet ( G , \allowbreak \cD , \allowbreak \varphi_0 ,
\allowbreak g_0 , \allowbreak \kappa , \allowbreak \delta ) $ is uniquely determined by the signature function 
$\TrSig:=\InSig\xi$, so $\InSig$ comes from some form $B$.

\subsection{The isomorphism theorem}

Let $\psi_0$ be an automorphism of $(\cD,\varphi_0)$ as a graded algebra with involution 
and let $ \psi_1 : \cV^{[g]} \rightarrow \cV' $ be a $\psi_0$-semilinear isomorphism of $G$-graded vector spaces.
The isomorphism $\psi_1$ becomes linear over $\cD$ if we regard it as a map from
$ ( \cV^{[g]} )^{ \psi_0^{-1} } $ to $\cV'$,
where the superscript $ \psi_0^{-1} $
refers to the twisted $\cD$-module structure:
for all $ v \in \cV $ and $ d \in \cD $,
$ v \cdot d := v \psi_0^{-1} (d) $.
Consequently, the map $ (v,w) \mapsto B'( \psi_1(v) , \psi_1(w) ) $,
which appears in the left-hand side of Equation \eqref{eq:relation_B},
is a $\varphi_0$-sesquilinear form on
$ ( \cV^{[g]} )^{ \psi_0^{-1} } $
that has the same parameters $(g_0',\kappa',\InSig')$
as the form $B'$ on $\cV'$.
Let us compute the parameters $(\underline{g_0},\underline{\kappa},\underline{\InSig})$ of the right-hand side,
\[ \underline{B} (v,w) := \psi_0( d_0 B(v,w) ) = \psi_0(d_0)\psi_0(B(v,w)), \]
regarded as a $\varphi_0$-sesquilinear form
on $ ( \cV^{[g]} )^{ \psi_0^{-1} } $,
in terms of the parameters $(g_0,\kappa,\InSig)$ of $B$ on $\cV$.

Since $ \cV_{hg}^{[g]} = \cV_h $ for all $ h \in G $, we have:
\begin{equation}\label{eq:Actg0}
\underline{g_0} = \deg \underline{B} = g^{-2} g_0 t_0
\end{equation}
where $ t_0 := \deg d_0 \in K$.
Next, for all $ h \in G $, we have 
$ \kappa(hT) = \dim_{\De} \cV_h $, so
\begin{equation}\label{eq:ActMult}
\underline{\kappa} (hT) =
\dim_{\De} \cV_{h}^{[g]} =
\dim_{\De} \cV_{g^{-1}h} =
\kappa (g^{-1}hT).
\end{equation}
Finally, we have to compute the signature of the $\De$-valued form
\[
X_{ \underline{g_0} h^2 }^{-1}
\underline{B} \vert_{ \cV_h^{[g]} \times \cV_h^{[g]} } =
X_{ t_0 t }^{-1} \psi_0(d_0) \psi_0( X_t )
 \psi_0 \big( X_t^{-1} B 
\vert_{ \cV_{g^{-1}h} \times \cV_{g^{-1}h} } \big) 
\]
where $ t := g^{-2} g_0 h^2 \in K $
(which is equivalent to $ t_0 t = \underline{g_0} h^2 $ being in $ K $).
Using the fact that the inertia of a $\De$-valued form is not affected by an automorphism of $\De$ (as an $\RR$-algebra), we get:
\begin{equation}\label{eq:ActInSig}
\underline{\InSig}(h) = \mathrm{sign}
\big( X_{ t_0 t }^{-1} \psi_0(d_0) \psi_0(X_t) \big)
\InSig( g^{-1} h ).
\end{equation}
Note that $X_{ t_0 t }^{-1} \psi_0(d_0) \psi_0(X_t) \in \RR^{\times} $,
so the sign above makes sense.
Indeed, as $d_0\in Z(\cD)$, we have $X_{ t_0 t }^{-1} \psi_0(d_0) \psi_0(X_t)\in\De\cap C_{\cD}(\De)=Z(\De)$, 
so the only case that needs attention is $ \De \cong \CC $.
Then, as the elements $ X_{ t_0 t } $, $\psi_0(d_0)$ and $\psi_0(X_t)$ commute with each other, and $\varphi_0\psi_0=\psi_0\varphi_0$, we have:
\begin{align*}
\varphi_0( X_{ t_0 t }^{-1} \psi_0(d_0) \psi_0(X_t) ) & =
\varphi_0( X_{ t_0 t }^{-1} ) \varphi_0( \psi_0(d_0) ) \varphi_0( \psi_0(X_t) ) =
\delta X_{ t_0 t }^{-1} \psi_0(d_0) \psi_0(\delta X_t)\\
& = X_{ t_0 t }^{-1} \psi_0(d_0) \psi_0(X_t),
\end{align*}
proving that $X_{ t_0 t }^{-1} \psi_0(d_0) \psi_0(X_t)\in\RR$.

We can express Equations \eqref{eq:Actg0}, \eqref{eq:ActMult} and \eqref{eq:ActInSig} in terms of group action.
Denote by $\mathrm{Aut}^G(\cD,\varphi_0)$ the group of automorphisms of $(\cD,\varphi_0)$ as a graded algebra with involution, 
and recall the subgroup $C\subseteq Z(\cD)^\times$ defined by Equation \eqref{def:groupC}. 
Then $\mathrm{Aut}^G(\cD,\varphi_0)$ acts naturally on $C$, 
so we can form their semidirect product $C\rtimes\mathrm{Aut}^G(\cD,\varphi_0)$.
Consider the collection of pairs $(\cV,B)$ where $\cV$ is a graded right $\cD$-module of finite dimension over $\cD$ 
and $B$ is a nondegenerate homogeneous $\varphi_0$-sesquilinear form on $\cV$ satisfying $\overline{B}=\delta B$.
The groups $\mathrm{Aut}^G(\cD,\varphi_0)$ and $C$ act on this collection as follows: for $\psi_0\in\mathrm{Aut}^G(\cD,\varphi_0)$, 
define $\psi_0\cdot(\cV,B) := (\cV^{\psi_0^{-1}},\psi_0 B)$ and, for $c\in C$, define $c\cdot(\cV,B) := (\cV,cB)$. 
Since $\psi_0(cB)=\psi_0(c)\psi_0 B$, this gives rise to an action of $C\rtimes\mathrm{Aut}^G(\cD,\varphi_0)$.  
Finally, $G$ acts by shift of grading: $g\cdot(\cV,B) := (\cV^{[g]},B)$, 
and this action commutes with that of $C\rtimes\mathrm{Aut}^G(\cD,\varphi_0)$. 
Then $((\cV^{[g]})^{\psi_0^{-1}},\underline{B})$ is obtained from $(\cV,B)$ by the action of
the element $(g,\psi_0(d_0),\psi_0)\in G\times(C\rtimes\mathrm{Aut}^G(\cD,\varphi_0))$.

Recall that if $d$ is a nonzero homogeneous element of $\cD$, $\psi_0' = \Int(d^{-1}) \psi_0 $ and 
$ \psi_1'(v) = \psi_1(v) d $ for all $ v \in \cV $, then $(\psi_0',\psi_1')$ and $(\psi_0,\psi_1)$ yield the same isomorphism 
$ \psi : \cR \rightarrow \cR' $.
In particular, it is sufficient to consider the subgroup $A\subseteq \mathrm{Aut}^G(\cD,\varphi_0)$
defined by Equation \eqref{def:groupA} or \eqref{def:groupA_}, depending on whether or not we are in the $\CC$-central case.

\begin{remark}\label{rem:Ain}
We can say more: let $\mathrm{Int}^G(\cD,\varphi_0)$ be the group of inner automorphisms of $(\cD,\varphi_0)$ as 
a graded algebra with involution, so, by Lemma \ref{lem:InnAut},  
$\mathrm{Int}^G(\cD,\varphi_0)=\{\Int(d) \mid d\in\Dgr^\times\text{ such that }\varphi_0(d) d \in Z(\cD)\}$. 
Then, $\mathrm{Int}^G(\cD,\varphi_0)$ acts trivially on $C$, and 
the $ G \times ( C \times \mathrm{Int}^G(\cD,\varphi_0) ) $-orbits in the set of isomorphism classes of pairs $(\cV,B)$ 
are the same as the $ G \times C $-orbits. 
\end{remark}

Equations \eqref{eq:Actg0}, \eqref{eq:ActMult} and \eqref{eq:ActInSig} show the effect of the above actions on the parameters 
$(g_0,\kappa,\InSig)\in\InParSet(G,\cD,\varphi_0,\delta)$ of $(\cV,B)$, where 
\[
\InParSet(G,\cD,\varphi_0,\delta) := \coprod_{g_0\in G}\coprod_{\kappa\in\MulSet(g_0)}\InSigSet(g_0,\kappa)\subseteq G\times\ZZ_{\ge 0}^{G/T}\times\ZZ^G,
\]
$ \MulSet(g_0) := \MulSet(G,\cD,\varphi_0,g_0,\delta) $ and 
$ \InSigSet(g_0,\kappa) := \InSigSet(G,\cD,\varphi_0,g_0,\kappa,\delta) $.
Recall that if we are not in the $\CC$-central case then $\psi_0\in A$ is given by $\psi_0(X_t d)=X_t \nu(t) d$ for all $t\in T$ and $d\in\De$, 
whereas if we are in the $\CC$-central case then either $\psi_0=\mathrm{id}_{\cD}$ or $\psi_0(X_t d)=X_t \nu(t) \bar{d}$ for all $t\in T$ and $d\in\De=\CC$, 
and also $C=\RR^\times$. 
Hence, $A$ acts on the group $C$ by $ \psi_0(c) = c\nu(\deg c) $ and 
on the set $ \InParSet(G,\cD,\varphi_0,\delta) $ through its action on each ``fiber'' $\InSigSet(g_0,\kappa)$:
\[
\psi_0 \cdot (g_0,\kappa,\InSig) = (g_0,\kappa,\underline{\InSig})
\text{ where }\underline{\InSig}(h)=\mathrm{sign}\big( \nu(g_0 h^2) \big) \InSig(h).
\]
(The factor $\mathrm{sign}( \nu(g_0 h^2) )$ is defined for $ g_0 h^2 \in K $, 
but $\InSig(h)=0$ for $ g_0 h^2 \notin K $, so the above formula makes sense.) 
Also, $C$ acts on the set $ \InParSet(G,\cD,\varphi_0,\delta) $ as follows:
\[
c \cdot (g_0,\kappa,\InSig) = (g_0 (\deg c), \kappa, \underline{\InSig})
\text{ where }
\underline{\InSig}(h) = \mathrm{sign} \big( X_{ g_0 h^2 (\deg c) }^{-1} X_{g_0 h^2} c \big) \InSig(h).
\]
(Again, the formula makes sense because $\InSig(h)=0$ for $ g_0 h^2 \notin K $.) 
Finally, $G$ acts on the set $ \InParSet(G,\cD,\varphi_0,\delta) $ through its componentwise action on $G\times\ZZ_{\ge 0}^{G/T}\times\ZZ^G$:
\[
g \cdot g_0 = g^{-2} g_0, \qquad
( g \cdot \kappa ) (hT) = \kappa (g^{-1}hT), \qquad
( g \cdot \InSig ) (h) = \InSig (g^{-1}h).
\]

The action of $G\times(C\rtimes A)$ can be reformulated in terms of the parameters $(g_0,\kappa,\TrSig) \in \ParSet(G,\cD,\varphi_0,\delta)$, 
where we use $ \SigSet(g_0,\kappa) := \SigSet ( G , \allowbreak \cD , \allowbreak \varphi_0 ,\allowbreak g_0 , \allowbreak \kappa , \allowbreak \delta ) $ 
instead of $\InSigSet(g_0,\kappa)$ to define the set $ \ParSet(G,\cD,\varphi_0,\delta) \subseteq G\times\ZZ_{\ge 0}^{G/T}\times\ZZ^{G/T}$. 
Then, for all $\psi_0\in A$, we have
\[
\psi_0 \cdot (g_0,\kappa,\TrSig) = (g_0,\kappa,\underline{\TrSig})
\text{ where }\underline{\TrSig}(x)=\mathrm{sign}\big( \nu(g_0 \xi(x)^2) \big) \TrSig(x),
\] 
and, similarly, for all $c\in C$, we have
\[
c \cdot (g_0,\kappa,\TrSig) = (g_0 (\deg c), \kappa, \underline{\TrSig})
\text{ where }
\underline{\TrSig}(x) = \mathrm{sign} \big( X_{ g_0 \xi(x)^2 (\deg c) }^{-1} X_{g_0 \xi(x)^2} c \big) \TrSig(x).
\]
Finally, Equation \eqref{eq:PropInSig} gives us a formula for the action of $g\in G$, which sends 
$\TrSig\in\SigSet(g_0,\kappa)$ to $\underline{\TrSig}\in\SigSet(g\cdot g_0, g\cdot\kappa)$:
\begin{align}
\underline{\TrSig}(x)
& = ( g \cdot \InSig )( \xi(x) )
= \InSig ( g^{-1} \xi(x) )
= \InSig ( \xi( g^{-1} x ) u )
\nonumber \\
& = \eta(u) \mathrm{sign} \big( X_{tu^2}^{-1} X_u X_t X_u \big)
\TrSig( g^{-1} x ) \label{eq:G-action}
\end{align}
where $x\in G/T$, $ t := g_0 \xi(g^{-1}x)^2 $ (which is in $K$ or else $\TrSig( g^{-1} x )=0$) 
and $ u := \xi( g^{-1} x )^{-1} g^{-1} \xi(x) $ (which is in $T$). 

\begin{theorem}\label{th:iso}
Let $M(\cD,\varphi_0,g_0,\kappa,\TrSig,\delta)$ and $M(\cD,\varphi_0,g_0',\kappa',\TrSig',\delta)$ be graded algebras 
with involution as in Definition \ref{def:Mdata}. If $\CC\cong\De\subseteq Z(\cD)$, assume that $T$ is finite and $\De=Z(\cD)$ 
(equivalently, $\beta$ is nondegenerate). Then the graded algebras with involution
are isomorphic if and only if $(g_0,\kappa,\TrSig)$ and $(g_0',\kappa',\TrSig')$
lie in the same $ G \times ( C \rtimes A ) $-orbit.
\end{theorem}

\begin{proof}
We have already proved that if $M(\cD,\varphi_0,g_0,\kappa,\TrSig,\delta)$ and $M(\cD,\varphi_0,g_0',\kappa',\TrSig',\delta)$ are isomorphic 
as graded algebras with involution then $(g_0,\kappa,\TrSig)$ and $(g_0',\kappa',\TrSig')$ lie in the same orbit.
Conversely, suppose that $(g_0',\kappa',\TrSig') = (g,c,\psi_0) \cdot (g_0,\kappa,\TrSig)$ and let $d_0=\psi_0^{-1}(c)$.
Then the $\varphi_0$-sesquilinear forms $\underline{B}=\psi_0(d_0 B)$ on $(\cV^{[g]})^{\psi_0^{-1}}$ and $B'$ on $\cV'$ are represented by 
the same matrix $\Phi$, as in Equation \eqref{eq:Phi}, with respect to appropriate homogeneous $\cD$-bases, whose respective elements have 
the same degrees. Therefore, there exists an isomorphism of graded $\cD$-modules $\psi_1:(\cV^{[g]})^{\psi_0^{-1}}\to\cV'$ such that 
$\underline{B}(v,w)=B'(\psi_1(v),\psi_1(w))$ for all $v,w\in\cV$. The result follows.
\end{proof}

\begin{remark}\label{rem:reduction_of_action}
If $K\ne T$ (that is, $\CC\cong\De\not\subseteq Z(\cD)$), 
then only the values $\nu(t)$ for $t\in K$ are relevant for the actions of $A$ 
on $C$ and on $\ParSet(G,\cD,\varphi_0,\delta)$, 
so we can replace $A\cong Z^1(T,\De^\times)$ by its quotient $\mathrm{Hom}^+(K,\RR^\times)$ 
(see Lemma \ref{lem:sses}). 
Further, only the sign of $\nu(t)$ matters, so we can replace $\mathrm{Hom}^+(K,\RR^\times)$ 
by its image under the homomorphism 
$\mathrm{Hom}(K,\RR^\times)\to\mathrm{Hom}(K,\{\pm 1\})$ induced by $\mathrm{sign}:\RR^\times\to\{\pm 1\}$, 
which is naturally isomorphic to 
$\mathrm{Hom}(K/T^{[2]},\{\pm 1\})$ where $T^{[2]} := \{t^2 \mid t\in T\}$. 
This latter reduction is also valid in the cases $\De=\RR$ and $\De\cong\HH$.
Therefore, except in the $\CC$-central case, the group $A$ can be replaced by 
$\bar{A} := \mathrm{Hom}(K/T^{[2]},\{\pm 1\})$.
Similarly, the group $C$ can always be replaced by $\bar{C} := C/\RR_{>0}$.
\end{remark}

\subsection{Interpretation in terms of groupoids}\label{sse:groupoids}

The classification we have just obtained can be expressed in the language of groupoids (that is, categories whose 
morphisms are invertible).
This interpretation will not be used in the remainder of the paper, but may elucidate what we have done so far.

For a given abelian group $G$, let $\mathfrak{R}$ be the groupoid whose objects are the pairs $(\cR,\varphi)$ as in Theorem \ref{th:struct} 
and whose morphisms are the isomorphisms of $G$-graded algebras with involution. Our classification problem is 
to parametrize the connected components of $\mathfrak{R}$. We begin by translating the problem to another groupoid,
$\mathfrak{M}$, defined as follows.

Let $\mathfrak{D}$ be the groupoid whose objects are the real graded division algebras with support in $G$ 
and the identity component isomorphic to $\RR$, $\CC$ or $\HH$ 
and whose morphisms are the isomorphisms of $G$-graded algebras. 
For a fixed object $\cD$ in $\mathfrak{D}$ and a degree-preserving involution $\varphi_0$ of $\cD$ that restricts to the conjugation on $\De$, 
let $\mathfrak{Q}(\cD,\varphi_0)$ be the groupoid whose objects are the pairs $(\cV,B)$, 
where $\cV$ is a graded right $\cD$-module of finite dimension over $\cD$ and 
$B:\cV\times\cV\to\cD$ is a nondegenerate homogeneous $\varphi_0$-sesquilinear form satisfying $\overline{B}\in\{\pm B\}$,
and whose morphisms $(\cV,B)\to(\cV',B')$ are the isomorphisms $\psi_1:\cV\to\cV'$ of graded $\cD$-modules such that 
$B=B'(\psi_1\times\psi_1)$.
Finally, let $\mathfrak{M}$ be the groupoid whose objects are the quadruples $(\cD,\varphi_0,\cV,B)$, 
where $\cD$ is an object in $\mathfrak{D}$, $\varphi_0$ is an involution of $\cD$ as above 
and $(\cV,B)$ is an object in $\mathfrak{Q}(\cD,\varphi_0)$, 
and whose morphisms $(\cD,\varphi_0,\cV,B)\to(\cD',\varphi_0',\cV',B')$ are the pairs 
$(\psi_0,\psi_1)$, where $\psi_0:\cD\to\cD'$ is a morphism in $\mathfrak{D}$ and there exist 
$g\in G$ and $d\in\Dgr^\times$ such that $\psi_1$ is a morphism 
$\big((\cV^{[g]})^{\psi_0^{-1}},\psi_0(dB)\big)\to(\cV',B')$ in $\mathfrak{Q}(\cD',\varphi_0')$, that is,
$\psi_1:(\cV^{[g]})^{\psi_0^{-1}}\to\cV'$ is an isomorphism of graded $\cD'$-modules and $\psi_0(dB)=B'(\psi_1\times\psi_1)$. 
Note that these conditions determine the elements $g$ and $d$ uniquely and imply that $\psi_0^{-1}\varphi_0'\psi_0=\Int(d)\varphi_0$. 

There is a functor $E:\mathfrak{M}\to\mathfrak{R}$ that maps 
$(\cD,\varphi_0,\cV,B)\mapsto(\mathrm{End}_\cD(\cV),\varphi)$ and $(\psi_0,\psi_1)\mapsto\psi$ 
where $\varphi$ is given by Equation \eqref{eq:GradAdj} and $\psi(r)=\psi_1 r\psi_1^{-1}$ for all $r\in\mathrm{End}_\cD(\cV)$.
The functor $E$ is full and essentially surjective (although not faithful, hence not an equivalence),
so it gives a bijection between the connected components of $\mathfrak{M}$ and those of $\mathfrak{R}$.

Next, we partition the groupoid $\mathfrak{M}$. Let $F_1:\mathfrak{M}\to\mathfrak{D}$ be the projection $(\cD,\varphi_0,\cV,B)\mapsto\cD$ and 
$(\psi_0,\psi_1)\mapsto\psi_0$. Then the object class of $\mathfrak{M}$ is partitioned into the inverse images of the connected components 
of $\mathfrak{D}$ under $F_1$. Moreover, if $\cD$ and $\cD'$ are in the same connected component of $\mathfrak{D}$, that is, there exists
a morphism $\psi_0:\cD\to\cD'$, and if $(\cD,\varphi_0,\cV,B)$ is an object in $\mathfrak{M}$, 
then $(\psi_0,\iota_\cV^{\psi_0^{-1}})$ is a morphism $(\cD,\varphi_0,\cV,B)\to (\cD',\psi_0\varphi_0\psi_0^{-1},\cV^{\psi_0^{-1}},\psi_0 B)$ in $\mathfrak{M}$,
where $\iota_\cV^{\psi_0^{-1}}$ is the identity map from $\cV$ to $\cV^{\psi_0^{-1}}$ (which have the same underlying set). 
It follows that every connected component of $\mathfrak{M}$ mapped by $F_1$ to the connected component $[\cD]$ in $\mathfrak{D}$ has a representative 
mapped to $\cD$. Therefore, we may fix $\cD$ and work in the ``fiber'' over $\cD$, that is, the full subgroupoid $\mathfrak{M}(\cD)\subseteq \mathfrak{M}$ 
defined by the object class $F_1^{-1}(\cD)$. Also, the group of automorphisms of $\cD$ acts on $F_1^{-1}(\cD)$, assuming the latter is nonempty: 
for $\psi_0\in\mathrm{Aut}^G(\cD)$, we let $\psi_0\cdot (\cD,\varphi_0,\cV,B) := (\cD,\psi_0\varphi_0\psi_0^{-1},\cV^{\psi_0^{-1}},\psi_0 B)$.

Further, we partition $\mathfrak{M}(\cD)$ as follows. The group $\Dgr^\times\rtimes\mathrm{Aut}^G(\cD)$ acts 
on the set of degree-preserving antiautomorphisms of $\cD$ by the formula $(c,\psi_0)\cdot\varphi_0 := \Int(c)\psi_0\varphi_0\psi_0^{-1}$,
so we get the corresponding action groupoid $\mathfrak{A}(\cD)$. Let $\mathfrak{I}(\cD)\subseteq\mathfrak{A}(\cD)$ 
be the full subgroupoid whose objects are the degree-preserving involutions of $\cD$ that restrict to the conjugation on $\De$. 
Then we can define another projection, $F_2:\mathfrak{M}(\cD)\to\mathfrak{I}(\cD)$, sending an object $(\cD,\varphi_0,\cV,B)$ to $\varphi_0$ and 
a morphism $(\psi_0,\psi_1):(\cD,\varphi_0,\cV,B)\to(\cD,\varphi_0',\cV',B')$ to the arrow $\varphi_0\to\varphi_0'$ labeled by the element 
$(\psi_0(d),\psi_0)\in\Dgr^\times\rtimes\mathrm{Aut}^G(\cD)$ where $d$ is determined by $(\psi_0,\psi_1)$ as above. 
Moreover, if $\varphi_0$ and $\varphi_0'$ are in the same connected component of $\mathfrak{I}(\cD)$, that is, 
there exist $c\in\Dgr^\times$ and $\psi_0\in\mathrm{Aut}^G(\cD)$ such that $(c,\psi_0)\cdot\varphi_0=\varphi_0'$ then 
we have $\psi_0^{-1}\varphi_0'\psi_0=\Int(d)\varphi_0$, where $d=\psi_0^{-1}(c)$, and 
the proof of Lemma \ref{lem:equiv_of_involutions} (with $\psi_0^{-1}\varphi_0'\psi_0$ playing the role of $\varphi_0'$) shows that $c$ 
can be chosen to satisfy $\varphi_0(d)\in\{\pm d\}$. 
Then $(\cV,dB)$ is an object of $\mathfrak{Q}(\cD,\Int(d)\varphi_0)$ and hence $(\psi_0,\iota_\cV^{\psi_0^{-1}})$ is a morphism 
$(\cD,\varphi_0,\cV,B)\to (\cD,\varphi_0',\cV^{\psi_0^{-1}},\psi_0 (dB))$ in $\mathfrak{M}(\cD)$.
Therefore, we may fix $\varphi_0$ and work in the ``fiber'' over $\varphi_0$, that is, the full subgroupoid 
$\mathfrak{M}(\cD,\varphi_0)\subseteq \mathfrak{M}(\cD)$ defined by the object class $F_2^{-1}(\varphi_0)$. 
Also, let $H$ be the subgroup of the stabilizer of $\varphi_0$ in $\Dgr^\times\rtimes\mathrm{Aut}^G(\cD)$ 
consisting of all $h=(\psi_0(d),\psi_0)$ such that $h\cdot\varphi_0=\varphi_0$ and $\varphi_0(d)\in\{\pm d\}$. 
Then, any morphism $(\psi_0,\psi_1)$ in $\mathfrak{M}(\cD,\varphi_0)$ is mapped by $F_2$ to an arrow $\varphi_0\to\varphi_0$ 
whose label is in $H$. Moreover, $H$ acts on $F_2^{-1}(\varphi_0)$, which is always nonempty, as follows: 
$h\cdot (\cD,\varphi_0,\cV,B) := (\cD,\varphi_0,\cV^{\psi_0^{-1}},\psi_0 (dB))$, and the corresponding action groupoid 
can be regarded as a subgroupoid of $\mathfrak{M}(\cD,\varphi_0)$: it has the same objects, 
and the arrow $(\cD,\varphi_0,\cV,B)\to(\cD,\varphi_0,\cV^{\psi_0^{-1}},\psi_0 (dB))$ labeled by $h$ can be identified 
with $(\psi_0,\iota_\cV^{\psi_0^{-1}})$.

Except in the case $\CC\cong\De\subseteq Z(\cD)$, all elements of $\mathrm{Aut}^G(\cD)$ commute with our involutions, 
so the situation simplifies: the connected components of $\mathfrak{I}(\cD)$ are given by the equivalence relation in 
Definition \ref{def:equiv_of_involutions} and the elements of $H$ have $d\in Z(\cD)$.
If $\varphi_0\vert_{Z(\cD)}=\mathrm{id}_{Z(\cD)}$ then there are no morphisms in $\mathfrak{M}(\cD,\varphi_0)$ between objects with 
$\overline{B}=+B$ and those with $\overline{B}=-B$, hence we partition $\mathfrak{M}(\cD,\varphi_0)$ into two full subgroupoids: 
$\mathfrak{M}(\cD,\varphi_0,\delta)$, $\delta\in\{\pm 1\}$, by the condition $\overline{B}=\delta B$.
If $\varphi_0\vert_{Z(\cD)}\ne\mathrm{id}_{Z(\cD)}$ then each connected component of $\mathfrak{M}(\cD,\varphi_0)$ has a representative 
in $\mathfrak{M}(\cD,\varphi_0,1)$, so it suffices to study this latter. In either case, the morphisms in $\mathfrak{M}(\cD,\varphi_0,\delta)$ are mapped 
to the subgroup $H^+\subseteq H$ defined by the condition $\varphi_0(d)=d$, that is,
\[
H^+ := \{ 
(\psi_0(d),\psi_0)\in\Dgr^\times\rtimes \mathrm{Aut}^G(\cD) \mid \varphi_0(d)=d
\text{ and } \psi_0\Int(d)\varphi_0\psi_0^{-1}=\varphi_0
\}.
\]
This subgroup acts on the objects of $\mathfrak{M}(\cD,\varphi_0,\delta)$, 
and the corresponding action groupoid can be regarded as a subgroupoid of $\mathfrak{M}(\cD,\varphi_0,\delta)$. 
In the case $\CC\cong\De\subseteq Z(\cD)$, each connected component of $\mathfrak{M}(\cD,\varphi_0)$ has a representative 
in $\mathfrak{M}(\cD,\varphi_0,1)$, so the same remarks about $H^+$ apply.
Note that, whenever all elements of $\mathrm{Aut}^G(\cD)$ commute with $\varphi_0$, 
we have $H^+=C\rtimes \mathrm{Aut}^G(\cD)$.

The lack of faithfulness of the functor $E$ can be exploited to replace the groupoid $\mathfrak{M}(\cD)$ 
by a subgroupoid with the same objects, but whose morphisms have some restrictions on $\psi_0$. 
This allowed us to replace the group $H^+$ by its subgroup $C\rtimes A$,
but we had to make a simplifying assumption in the case $\CC\cong\De\subseteq Z(\cD)$.
(Also, under the said assumption, the connected components of $\mathfrak{I}(\cD)$ are given by 
the equivalence relation in Definition \ref{def:equiv_of_involutions}.)
Now we proceed in full generality.

In addition to $H^+$, the group $G$ also acts on the objects of $\mathfrak{M}(\cD,\varphi_0,\delta)$: 
for all $g\in G$, let $g\cdot (\cD,\varphi_0,\cV,B) := (\cD,\varphi_0,\cV^{[g]},B)$, and the corresponding action groupoid 
can be regarded as a subgroupoid of $\mathfrak{M}(\cD,\varphi_0,\delta)$: it has the same objects, 
and the arrow $(\cD,\varphi_0,\cV,B)\to(\cD,\varphi_0,\cV^{[g]},B)$ labeled by $g$ can be identified with 
$(\mathrm{id}_\cD,\iota_\cV^{[g]})$, where $\iota_\cV^{[g]}$ is the identity map from $\cV$ to $\cV^{[g]}$ 
(which have the same underlying set). Since the actions of $G$ and $H^+$ commute, we get an action of $G\times H^+$,
and the corresponding action groupoid embeds in $\mathfrak{M}(\cD,\varphi_0,\delta)$.

Recall the groupoid $\mathfrak{Q}(\cD,\varphi_0)$. Its full subgroupoid $\mathfrak{Q}(\cD,\varphi_0,\delta)$,
determined by the condition $\overline{B}=\delta B$, also embeds in $\mathfrak{M}(\cD,\varphi_0,\delta)$, 
by sending $(\cV,B)\mapsto(\cD,\varphi_0,\cV,B)$ and $\psi_1\mapsto (\mathrm{id}_\cD,\psi_1)$.
For any morphism $\theta$ in $\mathfrak{M}(\cD,\varphi_0,\delta)$, we have (unique) factorizations: 
$\theta=\theta'\theta''=\tilde{\theta}''\tilde{\theta}'$ where $\theta'$ and $\tilde{\theta}'$ are morphisms in 
the action groupoid of $G\times H^+$, and $\theta''$ and $\tilde{\theta}''$ are morphisms in $\mathfrak{Q}(\cD,\varphi_0,\delta)$.
It follows that $G\times H^+$ acts on the connected components of $\mathfrak{Q}(\cD,\varphi_0,\delta)$, 
and each connected component of $\mathfrak{M}(\cD,\varphi_0,\delta)$ is the union of an orbit of this action.

It remains to parametrize the connected components of $\mathfrak{Q} := \mathfrak{Q}(\cD,\varphi_0,\delta)$, which we did by means of the set 
$\ParSet := \ParSet(G,\cD,\varphi_0,\delta)$ or, alternatively, $\InParSet := \InParSet(G,\cD,\varphi_0,\delta)$, and calculate the action of 
$G\times H^+$ (or its subgroup that has the same orbits) on $\ParSet$ (or $\InParSet$). The parametrization, that is, a mapping $P$ from 
$\ParSet$ (or $\InParSet$) to the object class of $\mathfrak{Q}$ that selects a unique representative in  
each connected component of $\mathfrak{Q}$, depends on the choice of a transversal for $T$ in $G$ and on the choice 
of the elements $X_t\in\cD_t$ for $t\in T$. The first is tantamount to selecting representatives for the isomorphism classes 
of graded right $\cD$-modules of dimension $1$ over $\cD$, and the second affects the matrix $\Phi$ representing $B$ in Theorem \ref{th:struct}.
Note that  $M(\cD,\varphi_0,g_0,\kappa,\TrSig,\delta)\cong E(P(g_0,\kappa,\TrSig))$.

To summarize: a parametrization of the isomorphism classes of $G$-graded algebras with involution as in Theorem \ref{th:struct} 
is given by $M(\cD,\varphi_0,g_0,\kappa,\TrSig,\delta)$, where $(g_0,\kappa,\TrSig)$ ranges over a set of representatives 
of the $G\times H^+$-orbits in $\ParSet(G,\cD,\varphi_0,\delta)$, $\cD$ ranges over a set of representatives of the 
isomorphism classes of real graded division algebras with $\De$ isomorphic to $\RR$, $\CC$ or $\HH$, 
$\varphi_0$ ranges over a set of representatives of the equivalence classes of degree-preserving involutions on $\cD$ 
that restrict to the conjugation on $\De$, and finally $\delta\in\{\pm 1\}$ if $\varphi_0$ restricts to the identity on $Z(\cD)$
and $\delta=1$ otherwise. The equivalence relation for the involutions is given by $\varphi_0\sim\varphi_0'$ if 
$\psi_0^{-1}\varphi_0'\psi_0=\Int(d)\varphi_0$ for some $\psi_0\in\mathrm{Aut}^G(\cD)$ 
and $d\in\Dgr^\times$ satisfying $\varphi_0(d)\in\{\pm d\}$.

\subsection{Central simple algebras with involution}\label{sse:consequences}

We now specialize to the setting needed for applications to classical central simple Lie algebras:
$M(\cD,\varphi_0,g_0,\kappa,\TrSig,\delta)$ is finite-dimensional and central simple as an ungraded algebra 
with involution. In other words, $\cD$ is finite-dimensional (hence semisimple because of the generalization of Maschke's Theorem),
its center $Z(\cD)$ is either $\RR$, $\CC$ or $ \tilde{\CC} := \RR \times \RR $,
and in the latter two cases $\varphi_0$ is an involution of the second kind, that is,
$ \varphi_0 \vert_{Z(\cD)} \neq \mathrm{id}_{Z(\cD)} $. 
Note that if $Z(\cD)=\tilde{\CC}$ then it must be nontrivially graded.
Also, in all cases we have $ C = \RR^{\times} $ and hence, by Remark \ref{rem:reduction_of_action},
we may replace $C$ with $ \{ \pm 1 \} $,
where $ \epsilon \in \{ \pm 1 \} $ acts on $\ParSet(G,\cD,\varphi_0,\delta)$ in the obvious way:
$ \epsilon \cdot ( g_0 , \kappa , \TrSig ) = ( g_0 , \kappa , \epsilon \TrSig ) $.

It is shown in \cite{BRpr} that $T$ must be an elementary $2$-group,
except possibly in the case $ \CC = \De = Z(\cD) $. We give a proof here for completeness and to introduce some notation for future use.
First, recall that $\rad\beta=\supp Z(\cD)$ by Lemma \ref{lem:suprad}. Since $\beta$ takes values in $\{\pm 1\}$ 
by Lemma \ref{lem:normbeta}, $t^2\in\rad\beta$ for all $t\in K$. 
Moreover, if $K\ne T$ then, for any $t\in T\setminus K$, we also have $t^2\in\rad\beta$ by Lemma \ref{lem:Xt2}, 
so we conclude that $t^2\in\rad\beta$ for all $t\in T$.
If $Z(\cD)$ is trivially graded, this immediately implies that $T$ is an elementary $2$-group.
Also note that $\beta$ is nondegenerate in this case.
If $Z(\cD)$ is nontrivially graded then $Z(\cD)$ is $\CC$ or $\tilde{\CC}$ and $\rad\beta=\supp Z(\cD)=\{e,f\}$ 
where $f\in K$ is an element of order $2$, 
which will be referred to as the \emph{distinguished element}. 
(Indeed, $f$ is the degree of an imaginary unit $\mathbf{i}$, 
where $\mathbf{i}^2=-1$ in the case of $\CC$ and $\mathbf{i}^2=1$ in the case of $\tilde{\CC}$.)
We claim that $f$ is not a square in $T$. Indeed, if $t^2=f$ then $X_t^2$ generates $Z(\cD)$ as an $\RR$-algebra and at the same time 
$\varphi_0(X_t^2)=\varphi_0(X_t)^2=\eta(t)^2 X_t^2=X_t^2$, which contradicts the fact $ \varphi_0 \vert_{Z(\cD)} \neq \mathrm{id}_{Z(\cD)} $.
It follows that $t^2=e$ for all $t\in T$, that is, $T$ is an elementary $2$-group. 

In particular, Remark \ref{rem:reduction_of_action} implies that $A\cong Z^1(T,Z(\De)^\times)$ 
(see Equation \eqref{def:groupA}) acts on $\ParSet(G,\cD,\varphi_0,\delta)$ 
through its quotient $\mathrm{Hom}(K,\{\pm 1\})$. To be precise, for $\nu\in\mathrm{Hom}(K,\{\pm 1\})$, we have
\[
\nu \cdot (g_0,\kappa,\TrSig) = (g_0,\kappa,\nu\TrSig)
\text{ where }(\nu\TrSig)(x) := \nu(g_0 \xi(x)^2)\TrSig(x).
\]
The same formula applies to the ``quadratic form'' $\nu:T\to\{\pm 1\}$ in the case  $ \CC = \De = Z(\cD) $,
if $T$ happens to be an elementary $2$-group (see Equation \eqref{def:groupA_}).

\begin{corollary}\label{cor:1}
Let $\cD$ be finite-dimensional with $ Z(\cD) = \RR $.
Then the graded algebras with involution $M(\cD,\varphi_0,g_0,\kappa,\TrSig,\delta)$ and $M(\cD,\varphi_0,g_0',\kappa',\TrSig',\delta)$ 
are isomorphic if and only if
$(g_0,\kappa,\TrSig)$ and $(g_0',\kappa',\TrSig')$
are in the same $ G \times \{ \pm 1 \} $-orbit.
\end{corollary}

\begin{proof}
By Skolem--Noether Theorem, we have $A\subseteq\mathrm{Int}^G(\cD,\varphi_0)$, so the result follows from Theorem \ref{th:iso}
and Remark \ref{rem:Ain}.
\end{proof}

In the next two corollaries, since $ \varphi_0 \vert_{Z(\cD)} \neq \mathrm{id}_{Z(\cD)} $, we necessarily have $\delta=1$.

\begin{corollary}\label{cor:2}
Suppose that $\cD$ is finite-dimensional and $Z(\cD)$ is $\CC$ or $\tilde{\CC}$, nontrivially graded.
Pick $ \nu \in \mathrm{Hom}( K , \{ \pm 1 \} ) $
such that $ \nu(f) = -1 $.
Then the graded algebras with involution $M(\cD,\varphi_0,g_0,\kappa,\TrSig,1)$ and $M(\cD,\varphi_0,g_0',\kappa',\TrSig',1)$
are isomorphic if and only if
$(g_0',\kappa',\TrSig')$ lies in the same $ G \times \{ \pm 1 \} $-orbit
as $(g_0,\kappa,\TrSig)$ or $(g_0,\kappa, \nu \TrSig )$.
\end{corollary}

\begin{proof}
Let $A_{\mathrm{in}}=A\cap\mathrm{Int}^G(\cD,\varphi_0)$. By Skolem--Noether Theorem, $\psi_0$ is inner 
if and only if $\psi_0\vert_{Z(\cD)}=\mathrm{id}_{Z(\cD)}$, which happens if and only if $\nu(f)=1$. 
Therefore, the group $A$ is generated by $A_{\mathrm{in}}$ and a single $\psi_0$ with $\nu(f)=-1$.
The result follows from Theorem \ref{th:iso} and Remark \ref{rem:Ain}.
\end{proof}

Finally, let us analyse the remaining case $ \CC = \De = Z(\cD) $.

\begin{corollary}\label{cor:3}
Suppose that $\cD$ is finite-dimensional and $Z(\cD)$ is $\CC$, trivially graded.
If $T$ is an elementary $2$-group, pick $ \nu : T \to \{ \pm 1 \} $ such that
$ \nu(st) = \nu(s) \nu(t) \beta(s,t) $ for all $ s,t \in T $.
Then the graded algebras with involution $M(\cD,\varphi_0,g_0,\kappa,\TrSig,1)$ and $M(\cD,\varphi_0,g_0',\kappa',\TrSig',1)$
are isomorphic if and only if
\begin{itemize}
\item when $T$ is not an elementary $2$-group:
$(g_0',\kappa',\TrSig')$ lies in the same $ G \times \{ \pm 1 \} $-orbit as $(g_0,\kappa,\TrSig)$;
\item when $T$ is an elementary $2$-group:
$(g_0',\kappa',\TrSig')$ lies in the same $ G \times \{ \pm 1 \} $-orbit
as $(g_0,\kappa,\TrSig)$ or $(g_0,\kappa, \nu \TrSig )$.
\end{itemize}
\qed
\end{corollary}

Thus, we see that, in the setting at hand, the isomorphism problem essentially boils down to the $G$-action on $\ParSet(G,\cD,\varphi_0,\delta)$.
Recall that the action is the following: 
$g\cdot(g_0,\kappa,\TrSig)=(g^{-2}g_0,g\cdot\kappa,g\cdot\TrSig)$ where $(g\cdot\kappa)(x)=\kappa(g^{-1}x)$ for all $x\in G/T$ 
and $g\cdot\TrSig$ is $\underline{\TrSig}$ given by Equation \eqref{eq:G-action}. 
(There is abuse of notation in writing $g\cdot\TrSig$ because, in general, $\underline{\TrSig}$ depends not only on $\TrSig$ but also on $g_0$.)
The sign factor in Equation \eqref{eq:G-action} can be simplified in our setting, but we first need to recall some facts about $\cD$ and $\varphi_0$.

When $Z(\cD)$ is $\CC$ with trivial grading, $\cD$ is a graded division algebra over $\CC$, so it is determined up to isomorphism of graded algebras 
by the pair $(T,\beta)$ where the alternating bicharacter $\beta:T\times T\to\CC^\times$ is nondegenerate 
(see e.g. \cite[Theorem 2.15]{EK13}).
In the remaining cases, we know that $T$ must be an elementary $2$-group and $\beta:K\times K\to\{\pm 1\}$. 

If $\De=\RR$ or $\De\cong\HH$ then $K=T$ and $X_t^2\in\RR$ for all $t\in T$, 
hence each $X_t$ can be normalized with a real scalar so that $X_t^2\in\{\pm 1\}$. 
This gives us a function $\mu:T\to\{\pm 1\}$ defined by $\mu(t)=X_t^2$. 
Moreover, each normalized element $X_t$ is unique up to sign, so the function $\mu$ is uniquely determined. 
It is shown in \cite[Theorems 15, 16 and 19]{R16}
(see \cite[\S 10]{BRpr} for the semisimple case)
that $\mu:T\to\{\pm 1\}$ is a quadratic form with polar form $\beta$
and that the pair $(T,\mu)$ determines the graded algebra $\cD$ up to isomorphism.

If $\De\cong\CC$ then $K\ne T$ (since $Z(\cD)$ is nontrivially graded) and, for any $t\in T\setminus K$, 
we have $X_t^2\in Z(\cD)_e=\RR$ (using Lemma \ref{lem:Xt2}). 
Moreover, for all $z\in\De$, we have $(X_tz)^2=X_t^2|z|^2$, hence each $X_t$ can be normalized 
so that $X_t^2\in\{\pm 1\}$, and this gives a uniquely determined function $\mu:T\setminus K\to\{\pm 1\}$. 
As to the elements $X_t$ with $t\in K$, we have $(X_tz)^2=X_t^2z^2$ for all $z\in \De$, 
hence they can also be normalized (with complex scalars) to satisfy $X_t^2\in\{\pm 1\}$, 
but the values $+1$ or $-1$ can be chosen arbitrarily. 
It is shown in \cite[Theorems 22 and 23]{R16}
(see \cite[\S 10]{BRpr} for the semisimple case)
that $\mu:T\setminus K\to\{\pm 1\}$ is what is called there a \emph{nice map},
that is, for some (and hence any) $u\in T\setminus K$, the function $\mu_u(t) := \mu(ut)\mu(u)$ is a quadratic form
$K\to\{\pm 1\}$ with polar form $\beta$, and that $(T,\mu)$ determines the graded algebra $\cD$ up to isomorphism.

Thanks to Skolem--Noether Theorem and Lemma \ref{lem:equiv_of_involutions}, 
all degree-preserving involutions on $\cD$ (of the second kind if $Z(\cD)$ is $\CC$ or $\tilde{\CC}$) are 
equivalent in the sense of Definition \ref{def:equiv_of_involutions}, hence $\varphi_0$ can be fixed arbitrarily. 
However, in all cases except $Z(\cD)=\tilde{\CC}$,
the classification in \cite{BRpr} suggests especially nice choices,
which will be referred to as the \emph{distinguished involutions}
(see \cite[\S 11]{BRpr}).
If $ Z(\cD) = \RR$ or $ Z(\cD) = \CC $ with nontrivial grading, then there is a unique
distinguished involution, whereas in the case $ Z(\cD) = \CC $ with trivial grading,
there is an isomorphism class of distinguished involutions.
They are defined as follows.

When $Z(\cD)$ is $\CC$ with trivial grading, we have $\De=Z(\cD)$ and $K=T$, 
and we agreed to choose $X_t$ so that $\eta(t)=1$ for all $t\in T$ 
(see Remark \ref{rem:delta} and Notation \ref{nota:scaling_Xt}) or, in other words, $\varphi_0(X_t)=X_t$ for all $t\in T$. 
Since $\varphi_0(X_tz)=X_t\bar{z}$ for all $z\in\De$, each element $X_t$ is determined up to a real scalar.
We can choose this scalar so that $|X_t^{o(t)}|=1$, where $o(t)$ denotes the order of $t$,
and this determines $X_t$ up to sign. Since $\varphi_0(X_t^{o(t)})=X_t^{o(t)}$, we have $X_t^{o(t)}\in\{\pm 1\}$.  
If $o(t)$ is odd then there is a unique choice of $X_t$ such that $X_t^{o(t)}=1$,
whereas if $o(t)$ is even then $X_t^{o(t)}$ is the same for both choices of $X_t$.
The distinguished involutions are characterized by the property that these normalized elements $X_t$
satisfy $X_t^{o(t)}=1$ for all $t\in T$. 

In the remaining cases, the distinguished involution is obtained by setting $\eta=\mu$. This requires some comment.
If $\De=\RR$ or $\De\cong\HH$ then $K=T$ and $\mu$ is a quadratic form with polar form $\beta$, hence 
the mapping $X_t\mapsto\mu(t)X_t$ uniquely extends to an involution $\varphi_0$ such that $\varphi_0\vert_{\De}$
is the conjugation.
If $\De\cong\CC$ (hence $K\ne T$), the nice map $\mu$ is defined only on $T\setminus K$. 
However, we agreed to choose the elements $X_t$ for $t\in K$ so that $\varphi_0(X_t)=\delta X_t$ 
(Notation \ref{nota:scaling_Xt}), which determines them up to a real scalar. 
Moreover, for these elements we have $\varphi_0(X_t^2)=X_t^2$, so $X_t^2\in\RR$ and hence $X_t$ can be normalized 
with real scalars so that $X_t^2\in\{\pm 1\}$. 
This determines a unique extension of $\mu$ to a function $T\to\{\pm 1\}$, which we still denote by $\mu$, by setting 
$\mu(t) := X_t^2$ for all $t\in T$. The distinguished involution is characterized by the property that $\eta(t)=\mu(t)$ 
for all $t\in T$; in fact, it is sufficient to require $\eta(t)=\mu(t)$ for all $t\in T\setminus K$ 
(see Lemma \ref{lem:extended_beta} below), that is, 
$\varphi_0(X_t)=\mu(t)X_t$ for all $t\in T\setminus K$.
It turns out that $\varphi_0$ thus defined is of the second kind if $Z(\cD)=\CC$, 
but of the first kind if $Z(\cD)=\tilde{\CC}$, so it is not suitable for our purposes in this latter case.

We will need one more ingredient in the case $K\ne T$, namely, the extension of 
$\beta:K\times K\to\{\pm 1\}$ to a function $T\times K\to\{\pm 1\}$, which we still denote by $\beta$.    
As we already observed, the elements $X_t$ with $t\in K$ are determined by the condition $\varphi_0(X_t)=\delta X_t$
up to a real scalar, hence Equation \eqref{def:beta} defines $\beta(s,t)\in\De$ uniquely as long as 
at least one of the elements $s$ and $t$ is in $K$.

\begin{lemma}\label{lem:extended_beta}
If $\De\cong\CC$ and $K\ne T$ then the function $\beta$ defined on $T\times K$ by the equation $X_uX_t=\beta(u,t)X_tX_u$
takes values in $\{\pm 1\}$ and is multiplicative in the first variable. It also satisfies
\[
\beta(u,t)=\mu(ut)\mu(u)\mu(t)=\eta(ut)\eta(u)\delta
\]
for all $u\in T\setminus K$ and $t\in K$. 
In particular, if $\eta$ and $\mu$ coincide on $T\setminus K$ then they coincide on the whole $T$. 
\end{lemma}

\begin{proof}
We already know that $\beta$ takes values in $\{\pm 1\}$ on $K\times K$, so let $u\in T\setminus K$ and $t\in K$. 
Applying $\varphi_0$ to $X_uX_t=\beta(u,t)X_tX_u$, we get
$X_tX_u=X_uX_t\overline{\beta(u,t)}=\beta(u,t)X_uX_t$, so $\beta(u,t)^2=1.$
Hence, $\beta(u,t)=\beta(u,t)^{-1}$ and we get $\Int(X_t)(X_u)=\beta(u,t)X_u$, so Lemma \ref{lem:cocycle} implies 
that $\beta(\cdot,t)\in Z^1(T,\{\pm 1\})=\mathrm{Hom}(T,\{\pm 1\})$.

We have $X_uX_t=X_{ut}\lambda$ for some $\lambda\in\De^\times$. 
Squaring both sides and taking into account that $ut\in T\setminus K$, we get 
$X_u^2X_t^2\beta(u,t)=X_{ut}^2|\lambda|^2$, hence $\mu(u)\mu(t)\beta(u,t)=\mu(ut)$.
Finally, applying $\varphi_0$ to $X_uX_t=X_{ut}\lambda$, we get 
$X_tX_u\eta(u)\eta(t)=\bar{\lambda}X_{ut}\eta(ut)=X_{ut}\lambda\eta(ut)$, hence 
$\eta(ut)=\beta(u,t)\eta(u)\eta(t)=\beta(u,t)\eta(u)\delta$.
\end{proof}

We will now obtain more explicit formulas for the action of $G$ on the set of parameters $\ParSet(G,\cD,\varphi_0,\delta)$.
We emphasize that, in the next result, whenever $K\ne T$, $\mu$ and $\beta$ stand for the extensions of, respectively, 
the nice map $\mu:T\setminus K\to\{\pm 1\}$ and the bicharacter $\beta:K\times K\to\{\pm 1\}$, as defined above.

\begin{proposition}\label{prop:G-action_simplification_nonCcentral}
Suppose that $\cD$ is finite-dimensional and $Z(\cD)$ is $\RR$ or either $\CC$ or $\tilde{\CC}$ with nontrivial grading. 
Then the action of $G$ on $\ParSet(G,\cD,\varphi_0,\delta)$ is the following:
$g\cdot(g_0,\kappa,\TrSig)=(g^{-2}g_0,g\cdot\kappa,g\cdot\TrSig)$
where $(g\cdot\kappa)(x) = \kappa(g^{-1}x)$ and 
\[
(g\cdot\TrSig)(x) = \eta(u)\mu(u)\beta(u,t)\TrSig(g^{-1}x),
\]
with $ u := \xi( g^{-1} x )^{-1} g^{-1} \xi(x) \in T$ and $ t := g_0 \xi(g^{-1}x)^2 = g^{-2} g_0 \xi(x)^2$ 
(which is in $K$ or else $\TrSig( g^{-1} x )=0$).
Furthermore, if $Z(\cD)$ is not $\tilde{\CC}$ and $\varphi_0$ is the distinguished involution then
\[
(g\cdot\TrSig)(x) = \beta(u,t)\TrSig(g^{-1}x)
=\beta(\xi( g^{-1} x )^{-1} g^{-1} \xi(x),g^{-2} g_0 \xi(x)^2)\,\TrSig(g^{-1}x).
\]
\end{proposition}

\begin{proof}
Taking into account that $T$ is an elementary $2$-group and using Lemma \ref{lem:extended_beta}, 
we can compute the sign factor in Equation \eqref{eq:G-action} as follows:
\[
\mathrm{sign}(X_{tu^2}^{-1}X_uX_tX_u)=\mathrm{sign}(X_t^{-1}X_uX_tX_u)=\beta(u,t)\mathrm{sign}(X_u^2)=\beta(u,t)\mu(u).
\]
Note that $\xi(g^{-1}x)g\xi(x)^{-1}$ is always in $T$ and hence $\xi(g^{-1}x)^2=g^{-2}\xi(x)^2$.
The result follows.
\end{proof}

If $ Z(\cD) = \CC = \De$ (hence $K=T$), the sign factor in Equation \eqref{eq:G-action} will, in general, depend on the choice  
of the elements $X_t$, because if $tu^2\ne t$ then $X_t$ and $X_{tu^2}$ can be independently normalized with real scalars. 
(This problem does not arise if $T$ happens to be an elementary $2$-group.) 
It is shown in \cite[\S 11]{BRpr} that, if $\varphi_0$ is a distinguished involution, then the elements $X_t$ can be chosen in such a way
that, in addition to the condition $X_t^{o(t)}=1$, they also satisfy $X_uX_sX_u=X_{su^2}$ for all $s\in T^{[2]}$ and $u\in T$.
Moreover, the elements $X_s$ for $s\in T^{[2]}$ are uniquely determined, while the elements $X_t$ for $t\notin T^{[2]}$ are determined up to sign.
We will make the following choice. Fix a transversal for the subgroup $T^{[2]}$ in $T$ containing $e$ and let $\xi':T/T^{[2]}\to T$ be the 
corresponding section of the quotient map $\pi':T\to T/T^{[2]}$. Also pick a transversal for the subgroup $T_{[2]}$ in $T$ containing $e$ 
and let $\xi'':T^{[2]}\to T$ be the corresponding section of the epimorphism $T\to T^{[2]}$ given by $t\mapsto t^2$ 
(in other words, $\xi''(s)^2=s$ for all $s\in T^{[2]}$, and $\xi''(e)=e$). 
Make an arbitrary choice of $X_t$ (out of two options) for each $t\ne e$ in the transversal of $T^{[2]}$, and then define
\begin{equation}\label{def:basis_Ccentral}
X_t := \beta(\xi''(s),\xi'\pi'(t))\,X_{\xi'\pi'(t)}X_s \text{ where } s := t(\xi'\pi'(t))^{-1}.
\end{equation}
Note that $\xi'\pi'(t)$ belongs to the transversal of $T^{[2]}$ and $s$ belongs to $T^{[2]}$, so the elements $X_{\xi'\pi'(t)}$ and $X_s$ 
were defined earlier. The above formula is consistent with those definitions because, for $t$ in the transversal of $T^{[2]}$, we have $\xi'\pi'(t)=t$ 
and hence $s=e$, while for $t\in T^{[2]}$, we have $\xi'\pi'(t)=e$ and hence $s=t$.

\begin{lemma}\label{lem:basis_Ccentral}
The elements defined by Equation \eqref{def:basis_Ccentral} satisfy $\varphi_0(X_t)=X_t$ and $X_t^{o(t)}=1$ for all $t\in T$.
Moreover, $X_u X_s X_u = X_{su^2}$ for all $s\in T^{[2]}$ and $u\in T$.
\end{lemma}

\begin{proof}
Using the fact $s=\xi''(s)^2$, we obtain:
\[
\begin{split}
\varphi_0(X_t) &= \overline{\beta(\xi''(s),\xi'\pi'(t)}\,X_s X_{\xi'\pi'(t)}\\
&=\beta(\xi''(s),\xi'\pi'(t)^{-1}\beta(s,\xi'\pi'(t))\,X_{\xi'\pi'(t)}X_s=X_t.
\end{split}
\]  
It is clear that $|X_t^{o(t)}|=1$, and the fact that $\varphi_0(X_t)=X_t$ implies $X_t^{o(t)}\in\{\pm 1\}$.
If $o(t)$ is odd then $t\in T^{[2]}$ and hence $X_t^{o(t)}=1$. If $o(t)$ is even then $X_t^{o(t)}=1$ because $\varphi_0$
is a distinguished involution. This completes the proof of the first assertion.
The second assertion follows by the choice of $X_s$ for $s\in T^{[2]}$.
\end{proof}

\begin{proposition}\label{prop:G-action_simplification_Ccentral}
Suppose that $\cD$ is finite-dimensional and $Z(\cD)$ is $\CC$ with trivial grading. 
Assume that $\varphi_0$ is a distinguished involution and the elements $X_t$ are defined by Equation \eqref{def:basis_Ccentral}.
Then the action of $G$ on $\ParSet(G,\cD,\varphi_0,1)$ is the following:
$g\cdot(g_0,\kappa,\TrSig)=(g^{-2}g_0,g\cdot\kappa,g\cdot\TrSig)$
where $(g\cdot\kappa)(x) = \kappa(g^{-1}x)$ and 
\[
(g\cdot\TrSig)(x) =
\beta(\xi''(su^2)^{-1}\xi''(s)u,\xi'\pi'(t))\,
\TrSig(g^{-1}x),
\]
with $ t := g_0 \xi(g^{-1}x)^2 $ (which is in $T$ or else $\TrSig( g^{-1} x )=0$),
$ u := \xi( g^{-1} x )^{-1} g^{-1} \xi(x) \in T $, and $s := t(\xi'\pi'(t))^{-1} \in T^{[2]}$.
\end{proposition}

\begin{proof}

Since $\pi'(tu^2)=\pi'(t)$, we have $tu^2(\xi'\pi'(tu^2))^{-1}=su^2$, and hence
\[
\begin{split}
& X_{tu^2}^{-1}X_uX_tX_u\\ 
&= (\beta(\xi''(su^2),\xi'\pi'(t))\,X_{\xi'\pi'(t)}X_{su^2})^{-1}X_u(\beta(\xi''(s),\xi'\pi'(t))\,X_{\xi'\pi'(t)}X_s)X_u\\
&= \beta(\xi''(su^2)^{-1}\xi''(s),\xi'\pi'(t))\, X_{su^2}^{-1} X_{\xi'\pi'(t)}^{-1} X_u X_{\xi'\pi'(t)} X_s X_u\\
&= \beta(\xi''(su^2)^{-1}\xi''(s)u,\xi'\pi'(t))\, X_{su^2}^{-1} X_{\xi'\pi'(t)}^{-1} X_{\xi'\pi'(t)} X_u X_s X_u\\
&= \beta(\xi''(su^2)^{-1}\xi''(s)u,\xi'\pi'(t)),
\end{split}
\]
in view of Lemma \ref{lem:basis_Ccentral}. Note that $\xi''(su^2)^{-1}\xi''(s)u\in T_{[2]}$ and hence 
\[
\beta(\xi''(su^2)^{-1}\xi''(s)u,\xi'\pi'(t))\in\{\pm 1\}.
\]
It remains to apply Equation \eqref{eq:G-action} and the fact $\eta(u)=1$.
\end{proof}

\begin{corollary}
If $T$ is an elementary $2$-group then $(g\cdot\TrSig)(x) = \beta(u,t)\,\TrSig(g^{-1}x)$.
If $|T|$ is odd then $(g\cdot\TrSig)(x) = \TrSig(g^{-1}x)$.
\end{corollary}

\begin{proof}
If $T$ is an elementary $2$-group then $T^{[2]}=\{e\}$ and hence $\xi'\pi'(t)=t$ for all $t\in T$. 
If $|T|$ is odd then $T^{[2]}=T$ and hence $\xi'\pi'(t)=e$ for all $t\in T$.
\end{proof}

\section{Classical central simple real Lie algebras}\label{se:Lie}

It is well known that a simple real Lie algebra is either a real form of a simple complex Lie algebra 
or a simple complex Lie algebra regarded as real. 
In the finite-dimensional case, the former algebras have centroid $\RR$ (that is, they are central)
and the latter have centroid $\CC$. 
Here we are interested in the real forms of classical simple complex Lie algebras. 
As recalled in the introduction, every such Lie algebra $\cL$ can be realized as $\mathrm{Skew}(\cR,\varphi)'$ 
where $(\cR,\varphi)$ is a semisimple finite-dimensional associative algebra with involution such that $\mathrm{Sym}(Z(\cR),\varphi)=\RR$,
that is, $(\cR,\varphi)$ is central simple as an algebra with involution.

Consider the affine group schemes $\mathbf{Aut}(\cL)$ and $\mathbf{Aut}(\cR,\varphi)$. The restriction map gives a homomorphism 
$\theta:\mathbf{Aut}(\cR,\varphi)\to\mathbf{Aut}(\cL)$, which is actually an isomorphism except when $\cL$ has type $A_1$ or $D_4$.
Indeed, since we are in characteristic $0$, it is sufficient to consider the homomorphism of the groups of points in the algebraic closure of the ground field,
namely, $\theta_\CC:\mathrm{Aut}_\CC(\cR\otimes_\RR\CC,\varphi\otimes\mathrm{id})\to \mathrm{Aut}_\CC(\cL\otimes_\RR \CC)$,
which is well known to be an isomorphism except for the said types. (It fails to be injective for type $A_1$ and surjective for type $D_4$.)

It follows (see e.g. Theorems 1.38 and 1.39 and also Remark 1.40 in \cite{EK13}) that the restriction gives a bijection between the abelian group gradings 
on $(\cR,\varphi)$ and on $\cL$, which induces a bijection between the isomorphism classes of $G$-gradings on $(\cR,\varphi)$ and on $\cL$, 
as well as between the equivalence classes of fine gradings.

We will not consider type $D_4$ in this paper. As to type $A_1$, we will use the alternative models: $\cL=\cR'$ where $\cR$ is $M_2(\RR)$ or $\HH$,
which give the two real forms of $\mathfrak{sl}_2(\CC)$, namely, $\mathfrak{sl}_2(\RR)$ and $\mathfrak{sl}_1(\HH)\cong\mathfrak{su}_2$.
There are analogous models for some real forms of type $A_r$ with $r>1$: $\cL=\cR'$ where $\cR$ is $M_n(\RR)$ with $n=r+1$ or $M_n(\HH)$ with $2n=r+1$.
These models give the real forms $\mathfrak{sl}_n(\RR)$ and $\mathfrak{sl}_n(\HH)$, respectively, and, 
since the restriction map $\mathbf{Aut}(\cR)\to\mathbf{Aut}(\cL)$ is a closed embedding with image $\mathbf{Int}(\cL)$, 
they will help us deal with the inner gradings on these real forms.
A $G$-grading on $\cL$ is said to be \emph{inner} if the image of the corresponding homomorphism $G^D\to\mathbf{Aut}(\cL)$ (see e.g. \cite[\S 1.4]{EK13})
is contained in $\mathbf{Int}(\cL)$; otherwise the grading is called \emph{outer}. In other words, the grading is inner if an only if the action of 
the character group $\widehat{G} := \mathrm{Hom}(G,\CC^\times)$ on the complexification $\cL_\CC := \cL\otimes_\RR \CC$ is by inner automorphisms.
(The action of the group $\widehat{G}$ on a $G$-graded complex algebra $\mathcal{A}=\bigoplus_{g\in G}\mathcal{A}_g$ is defined by  
$\chi\cdot a=\chi(g)a$ for all $\chi\in\widehat{G}$, $a\in\mathcal{A}_g$ and $g\in G$.) 
Therefore, if $\cL=\cR'$ as above, the inner $G$-gradings on $\cL$ are precisely the restrictions of the $G$-gradings on the associative algebra $\cR$. 
Of course, all gradings are inner if $\cL$ is of type $A_1$. 

In the models $\cL=\mathrm{Skew}(\cR,\varphi)'$ for type $A_r$ with $r>1$, 
the inner (respectively, outer) gradings on $\cL$ are the restrictions of the Type~I (respectively, Type~II) gradings on $(\cR,\varphi)$.
Note that $\cR$ is a graded-simple algebra in all cases except for the inner gradings on $\mathfrak{sl}_n(\RR)$ and $\mathfrak{sl}_n(\HH)$, 
which correspond to $(\cR,\varphi)$ isomorphic $M_n(\RR)\times M_n(\RR)$ and $M_n(\HH)\times M_n(\HH)$, respectively, with  
$\varphi:(X,Y)\mapsto (Y^*,X^*)$, where $X^* := \overline{X}^T$. In these latter cases, we are going to use the alternative models.

\subsection{The global signature}

In Subsection \ref{sse:consequences}, we classified up to isomorphism the finite-dimensional $G$-graded real associative algebras with involution 
$(\cR,\varphi)$ that are graded-simple and central simple as algebras with involution (disregarding the grading). 
This gives a classification up to isomorphism of the $G$-graded classical central simple real Lie algebras, except those mentioned above, 
by restricting the gradings from $\cR$ to $\cL=\mathrm{Skew}(\cR,\varphi)'$. However, we still have to identify the isomorphism class of $\cL$ 
in terms of our classification parameters. 

Recall the graded algebra with involution $(\cR,\varphi)=M(\cD,\varphi_0,g_0,\kappa,\TrSig,\delta)$ from Definition \ref{def:Mdata} 
and assume that $(\cD,\varphi_0)$ is central simple as an algebra with involution.
If we disregard the grading, then $(\cR,\varphi)$ is isomorphic to $M_k(\cD)$ with 
$\varphi(X)=\Phi^{-1}\varphi_0(X^T)\Phi$, where $k=|\kappa|$ and $\Phi$ is given by Equation \eqref{eq:Phi}. 
Hence, the type of $\cL$ is easy to identify: it is determined by the isomorphism class of $\cD$ as an ungraded algebra and, in the case $Z(\cR)=\RR$, 
also by $\delta$ and the type of $\varphi_0$ (orthogonal or symplectic). The isomorphism class of $\cD$ also gives us information about 
which real form of a given type we obtain, but in some cases when $\cD$ is simple, that is, $M_\ell(\Delta)$ with $\Delta\in\{\RR,\CC,\HH\}$, 
additional information is required to pin down the isomorphism class of $\cL$, namely, the signature of $\varphi$ as an involution on
$\cR\cong M_{k\ell}(\Delta)$. To be precise, if $\varphi$ is given in matrix form by $Z\mapsto\Psi^{-1} Z^* \Psi$ where 
$\Psi^*=\Psi$ (that is, $\Psi$ is a hermitian matrix, which is determined up to a real scalar) then the signature of 
$\varphi$ is defined as the absolute value of the signature of $\Psi$. We will refer to this parameter as the \emph{global signature}, 
and we will compute it now in terms of our signature function $\TrSig$ assuming that $\varphi_0$ is a distinguished involution
(see Subsection \ref{sse:consequences}).

We know from \cite[\S 11]{BRpr} that the involution $\varphi_0$ on $\cD$ 
is orthogonal if $\Delta=\RR$ and symplectic if $\Delta=\HH$. 
Hence, if we identify $\cD$ with $M_{\ell}(\Delta)$, then $ \varphi_0(Y) = \Lambda^{-1} Y^* \Lambda $
for all $ Y \in M_{\ell}(\Delta) $ where $ \Lambda \in M_{\ell}(\Delta) $ is hermitian.
The same is true in the case $\Delta=\CC$ because $\varphi_0$ is of the second kind.
Using the Kronecker product to identify
$M_k(\cD)$ with $ M_k(\RR) \otimes_{\RR} M_{\ell}(\Delta) $,
we obtain:
\[
\varphi( X \otimes Y )=\Psi^{-1}(X^T\otimes Y^*)\Psi
\]
for all $ X \in M_k(\RR) $ and $ Y \in M_{\ell}(\Delta) $,
where $\Psi=\Psi_1\oplus\dots\oplus\Psi_{m+r}$ with the blocks given by 
$\Psi_i = S_i\otimes\Lambda X_{t_i}$ unless $ \De \cong \HH $ and $ \eta(t_i) = -\delta $, 
and in this latter case 
$\Psi_i= I_{k_i} \otimes \Lambda\mathbf{i} X_{t_i}$ (recall that $\mathbf{i}$ and $X_{t_i}$ commute).

Note that $ S_i^T = \delta \eta(t_i) S_i $
and $ ( \Lambda X_{t_i} )^* 
= X_{t_i}^* \Lambda
= \eta(t_i) \Lambda X_{t_i} $,
hence a block $ S_i \otimes \Lambda X_{t_i} \in M_{ k \ell }(\Delta) $ 
is hermitian if and only if $ \delta = 1 $.
In the case $ \De \cong \HH $ and $ \eta(t_i) = -\delta $,
we have $ ( \Lambda \mathbf{i} X_{t_i} )^*
= \delta \Lambda \mathbf{i} X_{t_i} $,
because $ - \eta(t_i) \mathbf{i} X_{t_i}
= \varphi_0( \mathbf{i} X_{t_i} )
= \Lambda^{-1} (\mathbf{i} X_{t_i})^* \Lambda $.
Hence, the global signature is defined if and only if $\delta=1$.
So let us assume that $ \delta = 1 $ and compute the signature of the matrix $\Psi$ 
as the sum of the signatures of the blocks $\Psi_i$.

If $S_i$ is skew-symmetric,
then the bilinear form that it defines
has a totally isotropic subspace of maximal dimension,
so the same happens to the hermitian form defined by $ \Psi_i = S_i \otimes \Lambda X_{t_i} $,
which means that the signature is $0$.
If $S_i$ is symmetric, or if $ \De \cong \HH $ and $ \eta(t_i) = -1 $, then the first factor of $\Psi_i$ is real symmetric 
and the second is hermitian, hence the signature of $\Psi_i$ is the product of their signatures.
In this case, note that
the signature of the involution on $M_{\ell}(\Delta)$ given by 
$ Y \mapsto X_{t_i}^{-1} \Lambda^{-1} Y^* \Lambda X_{t_i} $
(respectively $ Y \mapsto X_{t_i}^{-1} \mathbf{i}^{-1} \Lambda^{-1}Y^* \Lambda \mathbf{i} X_{t_i} $)
is the absolute value of
the signature of $ \Lambda X_{t_i} $
(respectively $ \Lambda \mathbf{i} X_{t_i} $).
The signs depend on the choice of the elements $X_t$.

Assume now that we are not in the case $ \Delta = \CC = \De $.
If $ \De \cong \HH $ and $ \eta(t_i) = -1 $,
we know from \cite[\S 11]{BRpr} that the signature of
the involution on $M_{\ell}(\Delta)$ given by 
$Y \mapsto X_{t_i}^{-1} \mathbf{i}^{-1}\varphi_0(Y) \mathbf{i} X_{t_i} $ is $0$.
Hence we only have to consider the blocks of the form
$ S_i \otimes \Lambda X_{t_i} $.
We know, again from \cite[\S 11]{BRpr}, that the signature of
the involution on $M_{\ell}(\Delta)$
given by $Y \mapsto X_{t_i}^{-1} \varphi_0(Y) X_{t_i} $
is $\ell$ if $ t_i = e $,
and $0$ otherwise.
In particular, the signs mentioned above are the same for all terms that give nonzero contribution 
to the sum.
Therefore, the signature of $\Psi$ equals $ \pm \sum_{ t_i = e } ( p_i - q_i )\ell $ 
and hence the global signature is given by the following formula:
\begin{equation}\label{eq:global_signature_nonCcentral}
\mathrm{signature}(\varphi) 
= \ell\:\Big| \sum_{\substack{ x \in (G/T)_{g_0} \\ \tau(x) = e }}\TrSig(x) \Big|.
\end{equation}

In the case $ \Delta = \CC = \De $ (hence $ K = T $ and $ \eta = 1 $), 
we know from \cite[\S 11]{BRpr} that the distinguished involutions on $\cD$ are
those that have nonzero signature, namely,
$ \sqrt{ \vert T_{[2]} \vert } $.
Moreover, the signature of
the (second kind) involution on $M_{\ell}(\CC)$
given by $Y \mapsto X_t^{-1} \varphi_0(Y) X_t $ is
$ \sqrt{ \vert T_{[2]} \vert } $ if $ t \in T^{[2]} $,
and $0$ otherwise.
\emph{We choose $X_t$ for $t\in T^{[2]}$ in such a way 
that $X_uX_tX_u\in\RR_{>0}X_{tu^2}$ for all $u\in T$.}
It is shown in \cite[\S 11]{BRpr} that the same elements $X_t$, $t\in T^{[2]}$, 
have the property that $\mathrm{signature}(\Lambda X_t)=\mathrm{signature}(\Lambda)$. 
It follows that the signature of $\Psi$ equals
$\pm \sum_{ t_i \in T^{[2]} } ( p_i - q_i )
\sqrt{ \vert T_{[2]} \vert } $
and hence the global signature is given
by the following formula:
\begin{equation}\label{eq:global_signature_Ccentral}
\mathrm{signature}(\varphi) 
= \sqrt{ \vert T_{[2]} \vert } \; \Big|
\sum_{\substack{ x \in (G/T)_{g_0} \\ \tau(x) \in T^{[2]} }}
\TrSig(x)
\Big|.
\end{equation}

\subsection{Notation}\label{sse:notation}

Whenever possible (that is, except in the case $Z(\cD)=\tilde{\CC}$), we will use distinguished involutions $\varphi_0$. 
In our notation $M(\cD,\varphi_0, g_0,\kappa,\TrSig,\delta)$, 
we will substitute for $\cD$ the parameters that determine its isomorphism class 
(hence also the isomorphism class of $(\cD,\varphi_0)$ if $\varphi_0$ is distinguished) 
and omit the parameters that are clear from the context:
\begin{enumerate}
\item[(1)] If $Z(\cD)=\RR$ then we will write $M(\Delta_0,T,\mu,g_0,\kappa,\TrSig,\delta)$ 
where $\cD$ has the identity component $\Delta_0\in\{\RR,\CC,\HH\}$ and is determined by $(T,\mu)$, and $\varphi_0$ is the distinguished involution.
\item[(2)] If $Z(\cD)=\CC$, nontrivially graded, then we will write $M(\Delta_0,T,\mu,g_0,\kappa,\TrSig)$ 
where $\cD$ has the identity component $\Delta_0\in\{\RR,\CC,\HH\}$ and is determined by $(T,\mu)$, $\varphi_0$ is the distinguished involution, and $\delta=1$.
\item[(3)] If $Z(\cD)=\CC$, trivially graded, then we will write $M(T,\beta,g_0,\kappa,\TrSig)$ 
where $\cD$ has the identity component $\CC$ and is determined by $(T,\beta)$, 
$\varphi_0$ is a (fixed) distinguished involution, and $\delta=1$.
\item[(4)] If $Z(\cD)=\tilde{\CC}$ then we will write $M(\Delta_0,T,\mu,\eta,g_0,\kappa,\TrSig)$ 
where $\cD$ has the identity component $\Delta_0\in\{\RR,\CC,\HH\}$ and is determined by $(T,\mu)$, $\varphi_0$ is determined by $\eta$, and $\delta=1$.
\end{enumerate}
Recall that, except in Case (3), $T$ is an elementary $2$-group and $\mu$ is either a quadratic form $T\to\{\pm 1\}$ 
with polar form $\beta:T\times T\to\{\pm 1\}$ (for $\Delta_0\in\{\RR,\HH\}$) or a nice map $T\setminus K\to\{\pm 1\}$, 
whose associated quadratic forms $\mu_u:K\to\{\pm 1\}$, with $u\in T\setminus K$,
have the same polar form $\beta:K\times K\to\{\pm 1\}$ (for $\Delta_0=\CC$). 
The radical of $\beta$ is the support of the grading on $Z(\cD)$, which is 
the trivial subgroup $\{e\}$ in Cases (1) and (3) and the subgroup $\{e,f\}\subseteq K$ in Cases (2) and (4), 
where $f$ is the distinguished element. Note that, if $\mu$ is a quadratic form (respectively, nice map) then 
$\mu(f)=-1$ (respectively, $\mu_u(f)=-1$ for any $u\in T\setminus K$) in Case~(2) 
and $\mu(f)=+1$ (respectively, $\mu_u(f)=+1$ for any $u\in T\setminus K$) in Case~(4).

The isomorphism class of $\cD$ as an ungraded algebra is determined as follows, using the notation $\Arf(\mu)$ 
for the value $\epsilon\in\{\pm 1\}$ that $\mu$ takes more often, that is, $\Arf(\mu)=\epsilon$ if and only if 
$\vert\mu^{-1}(\epsilon)\vert > \vert\mu^{-1}(-\epsilon)\vert$. 
If $\mu$ takes values $-1$ and $+1$ equally often then $\Arf(\mu)$ is undefined.
It is defined in Cases (1) and (4) and boils down to the classical Arf invariant as follows. 
In Case (1), if $\mu$ is a quadratic form (respectively, nice map) then $\Arf(\mu)=(-1)^\alpha$ (respectively, $\Arf(\mu)=\mu(u)(-1)^\alpha$) 
where $\alpha$ is the Arf invariant in the field of order $2$ of the quadratic form $\mu$ (respectively, $\mu_u$).
In Case (4), the only difference is that we have to consider the quadratic forms induced from $\mu$ (respectively, $\mu_u$) modulo 
the radical $\{e,f\}$.
\begin{enumerate}
\item[(1)] If $Z(\cD)=\RR$ then $\cD\cong M_\ell(\Delta)$ where 
\begin{itemize}
\item for $\Delta_0\in\{\RR,\CC\}$: if $\Arf(\mu)=+1$ then $\Delta=\RR$ and $\ell^2=\vert T\vert\dim\Delta_0$, 
and if $\Arf(\mu)=-1$ then $\Delta=\HH$ and $\ell^2=\frac{1}{4}\vert T\vert\dim\Delta_0$;
\item for $\Delta_0=\HH$: if $\Arf(\mu)=+1$ then $\Delta=\HH$ and $\ell^2=\frac{1}{4}\vert T\vert\dim\Delta_0=\vert T\vert$, 
and if $\Arf(\mu)=-1$ then $\Delta=\RR$ and $\ell^2=\vert T\vert\dim\Delta_0=4\vert T\vert$.
\end{itemize}
\item[(2)] If $Z(\cD)=\CC$, nontrivially graded, then $\cD\cong M_\ell(\CC)$ where 
$\ell^2=\frac{1}{2}\vert T\vert\dim\Delta_0$.
\item[(3)] If $Z(\cD)=\CC$, trivially graded, then $\cD\cong M_\ell(\CC)$ where 
$\ell^2=\vert T\vert$.
\item[(4)] If $Z(\cD)=\tilde{\CC}$ then $\cD\cong M_\ell(\Delta)\times M_\ell(\Delta)$ where
\begin{itemize}
\item for $\Delta_0\in\{\RR,\CC\}$: if $\Arf(\mu)=+1$ then $\Delta=\RR$ and $\ell^2=\frac{1}{2}\vert T\vert\dim\Delta_0$, 
and if $\Arf(\mu)=-1$ then $\Delta=\HH$ and $\ell^2=\frac{1}{8}\vert T\vert\dim\Delta_0$;
\item for $\Delta_0=\HH$: if $\Arf(\mu)=+1$ then $\Delta=\HH$ and $\ell^2=\frac{1}{8}\vert T\vert\dim\Delta_0=\frac{1}{2}\vert T\vert$, 
and if $\Arf(\mu)=-1$ then $\Delta=\RR$ and $\ell^2=\frac{1}{2}\vert T\vert\dim\Delta_0=2\vert T\vert$.
\end{itemize} 
\end{enumerate}

Recall that $\kappa$ belongs to the set $\MulSet(G,\cD,\varphi_0,g_0,\delta)$ of admissible multiplicity functions 
$G/T\to\ZZ_{\ge 0}$ as in Definition \ref{def:MultFunct}, where we substitute 
$\De\cong\CC$ in Case (3) and $\De\cong\Delta_0$ in all other cases, $\delta=1$ except in Case (1), 
and also $\eta=1$ in Case (3) and $\eta=\mu$ in Cases (1) and (2). 
Note that, in Case (3), we have $K=T$ and the definition boils down to the following: $|\kappa|<\infty$ and 
$\kappa(g_0^{-1}x^{-1})=\kappa(x)$ for all $x\in G/T$. 
Also, $\TrSig$ belongs to the set $\SigSet(G,\cD,\varphi_0,g_0,\kappa,\delta)$ of signature functions as in Definition \ref{def:SignFunct}, 
where we make the same substitutions.

We will also need the graded matrix algebra $M_k(\cD)$ without involution, where $\cD$ is as in Case (1) and the 
grading is determined according to Equation \eqref{eq:MatrGrad} by an arbitrary multiplicity function 
$\kappa:G/T\to\ZZ_{\ge 0}$ satisfying $|\kappa|=k$. We will denote this graded algebra by $M(\Delta_0,T,\mu,\kappa)$.

The group $G$ acts on multiplicity functions in the natural way: $(g\cdot \kappa)(x) := \kappa(g^{-1}x)$ for all $x\in G/T$.
It also acts on the triples $(g_0,\kappa,\TrSig)\in\ParSet(G,\cD,\varphi_0,\delta)$ as described
by Proposition~\ref{prop:G-action_simplification_nonCcentral} in Cases (1), (2) and (4), and
by Proposition~\ref{prop:G-action_simplification_Ccentral} in Case (3).
Since the triples belonging to the same $G$-orbit produce isomorphic graded algebras with involution and 
since the action of $g\in G$ replaces $g_0$ by $g^{-2}g_0$, we can fix a transversal $\Theta$ for the subgroup $G^{[2]}$ in $G$
and insist that $g_0\in\Theta$. 

For a given $g_0$, its stabilizer $G_{[2]}$ acts on the set of pairs $(\kappa,\TrSig)$ where
$\kappa\in\MulSet(G,\cD,\varphi_0,g_0,\delta)$ and $\sigma\in\SigSet(G,\cD,\varphi_0,g_0,\kappa,\delta)$. 
Taking into account our substitutions, we will denote this set of pairs by $\ParSet(G,\Delta_0,T,\mu,g_0,\delta)$ in Case (1),
$\ParSet(G,\Delta_0,\allowbreak T,\allowbreak \mu,g_0)$ in Case (2), $\ParSet(G,T,g_0)$ in Case (3), and $\ParSet(G,\Delta_0,T,\eta,g_0)$ in Case (4).
The action of $G_{[2]}$ on this set is given by $g\cdot(\kappa,\TrSig)=(g\cdot\kappa,g\cdot\TrSig)$ where $g\cdot\kappa$ is as above,
but $g\cdot\TrSig$ is more complicated: the absolute value $|\TrSig|:G/T\to\ZZ_{\ge 0}$ is transformed in the same way as $\kappa$,
but there are sign changes given by Propositions \ref{prop:G-action_simplification_nonCcentral} and \ref{prop:G-action_simplification_Ccentral}.
For example, in Cases (1) and (2), the first of these propositions gives
\[
(g\cdot\TrSig)(x)=\beta(\xi( g^{-1} x )^{-1} g^{-1} \xi(x), \tau(x))\,\TrSig(g^{-1}x)
\]
for all $g\in G_{[2]}$ and $x\in(G/T)_{g_0}$. 
(Note that the set $(G/T)_{g_0}$ and the function $\tau:(G/T)_{g_0}\to T$ are invariant under $G_{[2]}$ 
and, by definition, $\TrSig(x)=0$ unless $x\in (G/T)_{g_0}$.)
Finally, the group $\{\pm 1\}$ acts on the pairs $(\kappa,\TrSig)$ by 
$\epsilon\cdot(\kappa,\TrSig) := (\kappa,\epsilon\TrSig)$ for $\epsilon\in\{\pm 1\}$.

\subsection{Series A: inner gradings on special linear Lie algebras}

Let $\cR=M(\Delta_0,T,\mu,\kappa)$ as defined above. The $G$-grading of $\cR$ restricts to an inner grading on 
its Lie subalgebra $\cR'$, which will be denoted by $\Gamma_{\mathfrak{sl}}^{\mathrm{(I)}}(\Delta_0,T,\mu,\kappa)$.
Fixing an isomorphism $\cD\cong M_\ell(\Delta)$, $\Delta\in\{\RR,\HH\}$, we may identify $\cR$ with $M_n(\Delta)$,
where $n=|\kappa|\ell$, and hence regard $\Gamma_{\mathfrak{sl}}^{\mathrm{(I)}}(\Delta_0,T,\mu,\kappa)$ as 
a grading on $\mathfrak{sl}_n(\Delta)$. Note that this identification depends on the choice 
of the isomorphism $\cD\cong M_\ell(\Delta)$, but the isomorphism class of 
the grading on $\mathfrak{sl}_n(\Delta)$ does not.

\begin{theorem}\label{th:AI_RH}
Let $\cL$ be one of the real special linear Lie algebras of type $A_r$, namely, 
$\mathfrak{sl}_{r+1}(\RR)$ or $\mathfrak{sl}_{(r+1)/2}(\HH)$ (if $r$ is odd). 
Then any inner $G$-grading on $\cL$ is isomorphic to $\Gamma_{\mathfrak{sl}}^{\mathrm{(I)}}(\Delta_0,T,\mu,\kappa)$ where 
$r=\vert\kappa\vert \sqrt{\vert T\vert \dim\Delta_0}-1$ and, 
in the first case, $\Arf(\mu)=1$ if $\Delta_0\in\{\RR,\CC\}$ and $\Arf(\mu)=-1$ if $\Delta_0=\HH$,
while in the second case, $\Arf(\mu)=-1$ if $\Delta_0\in\{\RR,\CC\}$ and $\Arf(\mu)=1$ if $\Delta_0=\HH$.
Moreover, two such gradings, $\Gamma_{\mathfrak{sl}}^{\mathrm{(I)}}(\Delta_0,T,\mu,\kappa)$
and $\Gamma_{\mathfrak{sl}}^{\mathrm{(I)}}(\Delta_0',T',\mu',\kappa')$,
are isomorphic if and only if
$ \Delta_0 = \Delta_0' $, $ T = T'$,
$ \mu = \mu' $, and $\kappa'$ is in the union of the $G$-orbits of $\kappa$ and $\bar{\kappa}$,
where $\bar{\kappa}(x):=\kappa(x^{-1})$ for all $x\in G/T$.
\end{theorem}

\begin{proof}
Any inner grading on $\cL$ uniquely extends to $M_n(\Delta)$ where 
$n=r+1$ if $\Delta=\RR$ and $2n=r+1$ if $\Delta=\HH$.
The resulting graded algebra is isomorphic to some $\cR=M(\Delta_0,T,\mu,\kappa)$ where the parameters
must satisfy the indicated conditions. Finally, two gradings on $\cL$ are isomorphic if and only if 
the corresponding graded algebras $\cR$ and $\cR'$ are either isomorphic or anti-isomorphic.
\end{proof}

\subsection{Series A: inner gradings on special unitary Lie algebras}

Let $(\cR,\varphi)=M(T,\beta,g_0,\kappa,\TrSig)$ as defined in Subsection \ref{sse:notation}, so 
$\cR\cong M_n(\CC)$ as an ungraded algebra, $n=|\kappa|\ell$, and $\varphi$ is of the second kind. 
The restriction of the $G$-grading of $\cR$ to its Lie subalgebra $\mathrm{Skew}(\cR,\varphi)'$ is 
inner and will be denoted by $\Gamma_{\mathfrak{su}}^{\mathrm{(I)}}(T,\beta,g_0,\kappa,\TrSig)$. 

\begin{theorem}\label{th:AI_C}
Let $\cL$ be one of the special unitary Lie algebras of type $A_r$ for $r\ge 2$, namely,
$\mathfrak{su}(p,q)$ where $p+q = r+1$, $p\ge q$. 
Then any inner $G$-grading on $\cL$ is isomorphic to $\Gamma_{\mathfrak{su}}^{\mathrm{(I)}}(T,\beta,g_0,\kappa,\TrSig)$ 
where $g_0\in\Theta$, $r=\vert\kappa\vert \sqrt{\vert T\vert}-1$, and
$p-q$ equals the right-hand side of Equation \eqref{eq:global_signature_Ccentral}. 
Moreover, two such gradings, $\Gamma_{\mathfrak{su}}^{\mathrm{(I)}}(T,\beta,g_0,\kappa,\TrSig)$
and $\Gamma_{\mathfrak{su}}^{\mathrm{(I)}}(T',\beta',g_0',\kappa',\TrSig')$,
are isomorphic if and only if
$ T = T'$, $ \beta = \beta' $, $ g_0 = g_0' $, and 
\begin{itemize}
\item when $T$ is not an elementary $2$-group: $(\kappa',\TrSig')$ is in the $G_{[2]}\times\{\pm 1\}$-orbit of
$(\kappa,\TrSig)$ in the set $\ParSet(G,T,g_0)$;
\item when $T$ is an elementary $2$-group: $(\kappa',\TrSig')$ is in the union of the $G_{[2]}\times\{\pm 1\}$-orbits 
of $(\kappa,\TrSig)$ and $(\kappa,\nu\TrSig)$,
where $\nu:T\to\{\pm 1\}$ is a quadratic form with polar form $\beta$ and 
$(\nu\TrSig)(x) := \nu(\tau(x))\TrSig(x)$ for all $x\in G/T$.
\end{itemize}
\end{theorem}

\begin{proof}
Any inner grading on $\cL$ uniquely extends to a grading on $M_n(\CC)$ compatible with the (second kind) 
involution defining $\cL$ and giving the center $\CC$ the trivial grading. 
By Theorem \ref{th:struct}, the resulting graded algebra with involution must be isomorphic to some 
$M(T,\beta,g_0,\kappa,\TrSig)$, since we are in Case (3) of Subsection \ref{sse:notation}.
This proves the first assertion. The second assertion follows from Corollary \ref{cor:3}. 
\end{proof}

\subsection{Series A: outer gradings on special linear Lie algebras}

Let $(\cR,\varphi)=M(\Delta_0,T,\mu,\eta,g_0,\kappa,\TrSig)$ as defined in Subsection \ref{sse:notation}, so 
$\cR\cong M_n(\Delta)\times M_n(\Delta)$ as an ungraded algebra, $\Delta\in\{\RR,\HH\}$, $n=|\kappa|\ell$, 
and $\varphi$ is of the second kind. Recall that we may use an arbitrary second kind involution $\varphi_0$ on $\cD$.
\emph{For each $(\Delta_0,T,\mu)$, we fix a quadratic form (respectively, nice map) $\eta$ if $\Delta_0\in\{\RR,\HH\}$
(respectively, if $\Delta_0=\CC$) such that $\eta(f)=-1$ (respectively, $\eta_u(f)=-1$).}
The restriction of the $G$-grading of $\cR$ to its Lie subalgebra 
$\mathrm{Skew}(\cR,\varphi)'\cong\mathfrak{sl}_n(\Delta)$ is 
outer and will be denoted by $\Gamma_{\mathfrak{sl}}^{\mathrm{(II)}}(\Delta_0,T,\mu,\eta,g_0,\kappa,\TrSig)$. 

\begin{theorem}\label{th:AII_RH}
Let $\cL$ be one of the real special linear Lie algebras of type $A_r$ for $r\ge 2$, namely, 
$\mathfrak{sl}_{r+1}(\RR)$ or $\mathfrak{sl}_{(r+1)/2}(\HH)$ (if $r$ is odd). 
Then any outer $G$-grading on $\cL$ is isomorphic to 
$\Gamma_{\mathfrak{sl}}^{\mathrm{(II)}}(\Delta_0,T,\mu,\eta,g_0,\kappa,\TrSig)$ where 
$g_0\in\Theta$, $r=\vert\kappa\vert \sqrt{\frac{1}{2}\vert T\vert \dim\Delta_0}-1$ and, 
in the first case, $\Arf(\mu)=1$ if $\Delta_0\in\{\RR,\CC\}$ and $\Arf(\mu)=-1$ if $\Delta_0=\HH$,
while in the second case, $\Arf(\mu)=-1$ if $\Delta_0\in\{\RR,\CC\}$ and $\Arf(\mu)=1$ if $\Delta_0=\HH$.
Moreover, two such gradings, $\Gamma_{\mathfrak{sl}}^{\mathrm{(II)}}(\Delta_0,T,\mu,\eta,g_0,\kappa,\TrSig)$
and $\Gamma_{\mathfrak{sl}}^{\mathrm{(II)}}(\Delta_0',T',\mu',\eta',g_0',\kappa',\TrSig')$,
are isomorphic if and only if
$ \Delta_0 = \Delta_0' $, $ T = T'$,
$ \mu = \mu' $ (hence $\eta=\eta'$), $g_0=g_0'$, and
$(\kappa',\TrSig')$ is in the union of the $G_{[2]}\times\{\pm 1\}$-orbits 
of $(\kappa,\TrSig)$ and $(\kappa,\nu\TrSig)$ in the set $\ParSet(G,\Delta_0,T,\eta,g_0)$,
where $\nu:K\to\{\pm 1\}$ is a homomorphism satisfying $\nu(f)=-1$ and 
$(\nu\TrSig)(x) := \nu(\tau(x))\TrSig(x)$ for all $x\in G/T$.
\end{theorem}

\begin{proof}
Any outer grading on $\cL$ uniquely extends to a grading on $M_n(\Delta)\times M_n(\Delta)$ 
compatible with the (second kind) involution defining $\cL$ and giving the center $\tilde{\CC}$ a nontrivial grading. 
By Theorem \ref{th:struct}, the resulting graded algebra with involution must be isomorphic to some 
$M(\Delta_0,T,\mu,\eta,g_0,\kappa,\TrSig)$, since we are in Case (4) of Subsection \ref{sse:notation}.
This proves the first assertion. The second assertion follows from Corollary \ref{cor:2}. 
\end{proof}

\subsection{Series A: outer gradings on special unitary Lie algebras}

Let $(\cR,\varphi)=M(\Delta_0,T,\mu,g_0,\kappa,\TrSig)$ as defined in Subsection \ref{sse:notation}, so 
$\cR\cong M_n(\CC)$ as an ungraded algebra, $n=|\kappa|\ell$, 
and $\varphi$ is of the second kind.
The restriction of the $G$-grading of $\cR$ to its Lie subalgebra $\mathrm{Skew}(\cR,\varphi)'$ is 
outer and will be denoted by $\Gamma_{\mathfrak{su}}^{\mathrm{(II)}}
( \Delta_0 , \allowbreak T , \allowbreak \mu, \allowbreak
g_0 , \allowbreak \kappa , \allowbreak \TrSig)$.

The proof of the next result is analogous to Theorem \ref{th:AII_RH}, with Case (2) instead of Case (4).

\begin{theorem}\label{th:AII_C}
Let $\cL$ be one of the special unitary Lie algebras of type $A_r$ for $r\ge 2$, namely,
$\mathfrak{su}(p,q)$ where $p+q = r+1$, $p\ge q$. 
Then any outer $G$-grading on $\cL$ is isomorphic to 
$\Gamma_{\mathfrak{su}}^{\mathrm{(II)}}(\Delta_0,T,\mu,g_0,\kappa,\TrSig)$ 
where $g_0\in\Theta$, $r=\vert\kappa\vert \sqrt{\frac{1}{2}\vert T\vert\dim\Delta_0}-1$, and
$p-q$ equals the right-hand side
of Equation \eqref{eq:global_signature_nonCcentral}. 
Moreover, two such gradings, $\Gamma_{\mathfrak{su}}^{\mathrm{(II)}}(\Delta_0,T,\mu,g_0,\kappa,\TrSig)$
and $\Gamma_{\mathfrak{sl}}^{\mathrm{(II)}}(\Delta_0',T',\mu',g_0',\kappa',\TrSig')$,
are isomorphic if and only if
$ \Delta_0 = \Delta_0' $, $ T = T'$,
$ \mu = \mu' $, $ g_0 = g_0' $, and
$(\kappa',\TrSig')$ is in the union of the $G_{[2]}\times\{\pm 1\}$-orbits 
of $(\kappa,\TrSig)$ and $(\kappa,\nu\TrSig)$ in the set $\ParSet(G,\Delta_0,T,\mu,g_0)$,
where $\nu:K\to\{\pm 1\}$ is a homomorphism satisfying $\nu(f)=-1$ and 
$(\nu\TrSig)(x) := \nu(\tau(x))\TrSig(x)$ for all $x\in G/T$.
\qed
\end{theorem}

\subsection{Series B}

For series $B$, $C$ and $D$, we deal with central simple associative algebras over $\RR$, so we are in Case (1) 
of Subsection \ref{sse:notation}. The special feature of series $B$ is that the degree of this algebra is odd, 
so it must be $M_n(\RR)$ with odd $n$. This simplifies the situation dramatically. 

Let $(\cR,\varphi)=M(\Delta_0,T,\mu,g_0,\kappa,\TrSig,\delta)$ such that $\cR\cong M_n(\RR)$ as an ungraded algebra  
and $n=|\kappa|\ell$ is odd. We must have $\Delta_0=\RR$ and, since $T$ is an elementary $2$-group, 
we must also have $T=\{e\}$. The involution is orthogonal, so $\delta=1$. 
Moreover, since $|\kappa|$ is odd and $\kappa:G\to\ZZ_{\ge 0}$ is an admissible multiplicity function, 
part (b) of Definition \ref{def:MultFunct} implies that $g_0g^2=e$ for some $g\in G$. 
Therefore, we can make $g_0=e$ using the action of $G$.
Note that the definition of signature function $\TrSig:G\to\ZZ$ becomes: $\TrSig(x)=0$ for all $x\notin G_{[2]}$,
while $|\TrSig(x)|\le\kappa(x)$ and $\TrSig(x)\equiv\kappa(x)\pmod{2}$ for all $x\in G_{[2]}$. 
The action of $G_{[2]}$ on signature functions is just
$(g\cdot\TrSig)(x)=\TrSig(g^{-1}x)$ for all $x\in G$.

The restriction of the $G$-grading of $M(\RR,\{e\},1,e,\kappa,\TrSig,1)$, with odd $|\kappa|$, to its Lie subalgebra 
$\mathrm{Skew}(\cR,\varphi)$ will be denoted by $\Gamma_B(\kappa,\TrSig)$.

\begin{theorem}\label{th:B}
Let $\cL$ be one of the real forms of type $B_r$ for $r\ge 2$, namely, $\mathfrak{so}_{p,q}(\RR)$ 
where $p+q = 2r+1$, $p\ge q$.
Then any $G$-grading on $\cL$ is isomorphic to
$\Gamma_B(\kappa,\TrSig)$
where $ r = \frac{1}{2}(\vert \kappa \vert - 1) $ and 
$p-q=|\sum_{g\in G_{[2]}}\sigma(g)|$. 
Moreover, two such gradings, $\Gamma_B(\kappa,\TrSig)$ and $\Gamma_B(\kappa',\TrSig')$,
are isomorphic if and only if
$(\kappa,\TrSig)$ and $(\kappa',\TrSig')$
are in the same $ G_{[2]} \times \{ \pm 1 \} $-orbit.  
\end{theorem}

\begin{proof}
Any grading on $\cL$ uniquely extends to a grading on $M_{2r+1}(\RR)$ 
compatible with the involution defining $\cL$. 
By Theorem \ref{th:struct} and the above discussion, the resulting graded algebra with involution 
must be isomorphic to $M(\RR,\{e\},1,e,\kappa,\TrSig,1)$ for some $\kappa$ and $\TrSig$. 
Note that $p-q=\mathrm{signature}(\varphi)$ is given by Equation \eqref{eq:global_signature_nonCcentral},
which simplifies in view of the fact $T=\{e\}$. 
This proves the first assertion. The second assertion follows from Corollary \ref{cor:1}. 
\end{proof}

\begin{remark}
This result is valid for $r=1$ as well, and gives another parametrization of gradings for the real forms of type $A_1$:
$\mathfrak{so}_{3,0}(\RR)\cong\mathfrak{su}(2)\cong\mathfrak{sl}_1(\HH)$ and $\mathfrak{so}_{2,1}(\RR)\cong\mathfrak{su}(1,1)\cong\mathfrak{sl}_2(\RR)$.
\end{remark}

\subsection{Series C}

Let $(\cR,\varphi)=M(\Delta_0,T,\mu,g_0,\kappa,\TrSig,\delta)$ as defined in Subsection \ref{sse:notation}, 
where $\Arf(\mu)=-\delta$ if $\Delta_0\in\{\RR,\CC\}$ 
and $\Arf(\mu)=\delta$ if $\Delta_0=\HH$, so that $\varphi$ is a symplectic involution.
The restriction of the $G$-grading of $\cR$ to its Lie subalgebra $\mathrm{Skew}(\cR,\varphi)$ will be denoted by 
$\Gamma_C(\Delta_0,T,\mu,g_0,\kappa,\TrSig,\delta)$. 

\begin{theorem}\label{th:C}
Let $\cL$ be one of the real forms of type $C_r$ for $r\ge 2$, namely, $\mathfrak{sp}_{2r}(\RR)$ 
or $\mathfrak{sp}(p,q)$ where $p+q = r$, $p\ge q$.
Then any $G$-grading on $\cL$ is isomorphic to
$\Gamma_C(\Delta_0,T,\mu,g_0,\kappa,\TrSig,\delta)$
where $ g_0\in\Theta $, $ r = \frac{1}{2}\vert \kappa \vert \sqrt{\vert T\vert\dim\Delta_0} $ and, in the first case,
$\delta=-1$, while in the second case, $ \delta = 1 $ and
$p-q$ equals the right-hand side
of Equation \eqref{eq:global_signature_nonCcentral}.
Moreover, two such gradings, $\Gamma_C(\Delta_0,T,\mu,g_0,\kappa,\TrSig,\delta)$
and $\Gamma_C(\Delta_0',T',\mu',g_0',\kappa',\TrSig',\delta')$,
are isomorphic if and only if
$ \Delta_0 = \Delta_0' $, $ T = T'$,
$ \mu = \mu' $, $ g_0 = g_0' $, $ \delta = \delta' $, and
$(\kappa,\TrSig)$ and $(\kappa',\TrSig')$
are in the same $ G_{[2]} \times \{ \pm 1 \} $-orbit in the set $\ParSet(G,\Delta_0,T,\mu,g_0,\delta)$. 
\end{theorem}

\begin{proof}
If we give the algebra $M_{2r}(\RR)$ or $M_r(\HH)$ a $G$-grading that is compatible with a fixed symplectic involution, then,
by Theorem \ref{th:struct}, the resulting graded algebra with involution must be isomorphic to some $M(\Delta_0,T,\mu,g_0,\kappa,\TrSig,\delta)$ as above.
If $\delta=-1$ then $\Arf(\mu)=1$ for $\Delta_0\in\{\RR,\CC\}$ and $\Arf(\mu)=-1$ for $\Delta_0=\HH$, 
hence $\cD\cong M_\ell(\RR)$ as an ungraded algebra, where $\ell=\sqrt{\vert T\vert\dim\Delta_0}$.
This implies that $\mathrm{Skew}(\cR,\varphi)\cong\mathfrak{sp}_{k\ell}(\RR)$ where $k=\vert\kappa\vert$.
Similarly, if $\delta=1$ then $\cD\cong M_\ell(\HH)$ as an ungraded algebra, where $\ell=\frac{1}{2}\sqrt{\vert T\vert\dim\Delta_0}$, 
and this implies $\mathrm{Skew}(\cR,\varphi)\cong\mathfrak{sp}(p,q)$ where $p+q=k\ell$ and $p-q=\mathrm{signature}(\varphi)$.
This proves the first assertion. The second assertion follows from Corollary \ref{cor:1}. 
\end{proof}

\begin{remark}
This result is valid for $r=1$ as well, and gives yet another parametrization of gradings for the real forms of type $A_1$:
$\mathfrak{sp}_2(\RR)=\mathfrak{sl}_2(\RR)$ and $\mathfrak{sp}(1,0)=\mathfrak{sl}_1(\HH)$.
We also get another parametrization of gradings for type $B_2=C_2$.
\end{remark}

\subsection{Series D}

Let $(\cR,\varphi)=M(\Delta_0,T,\mu,g_0,\kappa,\TrSig,\delta)$ as defined in Subsection \ref{sse:notation}, 
where $\Arf(\mu)=\delta$ if $\Delta_0\in\{\RR,\CC\}$ 
and $\Arf(\mu)=-\delta$ if $\Delta_0=\HH$, so that $\varphi$ is an orthogonal involution.
The restriction of the $G$-grading of $\cR$, with even $|\kappa|\ell$ if $\delta=1$, 
to its Lie subalgebra $\mathrm{Skew}(\cR,\varphi)$ will be denoted by 
$\Gamma_D(\Delta_0,T,\mu,g_0,\kappa,\TrSig,\delta)$. 

The proof of the next result is completely analogous to Theorem \ref{th:C}.

\begin{theorem}\label{th:D}
Let $\cL$ be one of the real forms of type $D_r$ for $r=3$ or $r\ge 5$, namely, $\mathfrak{u}^*(r)$ 
or $\mathfrak{so}_{p,q}(\RR)$ where $p+q = 2r$, $p\ge q$.
Then any $G$-grading on $\cL$ is isomorphic to
$\Gamma_D(\Delta_0,T,\mu,g_0,\kappa,\TrSig,\delta)$
where $ g_0\in\Theta $, $ r = \frac{1}{2}\vert \kappa \vert \sqrt{\vert T\vert\dim\Delta_0} $ and, in the first case,
$\delta=-1$, while in the second case, $ \delta = 1 $ and
$p-q$ equals the right-hand side
of Equation \eqref{eq:global_signature_nonCcentral}.
Moreover, two such gradings, $\Gamma_D(\Delta_0,T,\mu,g_0,\kappa,\TrSig,\delta)$
and $\Gamma_D(\Delta_0',T',\mu',g_0',\kappa',\TrSig',\delta')$,
are isomorphic if and only if
$ \Delta_0 = \Delta_0' $, $ T = T'$,
$ \mu = \mu' $, $ g_0 = g_0' $, $ \delta = \delta' $, and
$(\kappa,\TrSig)$ and $(\kappa',\TrSig')$
are in the same $ G_{[2]} \times \{ \pm 1 \} $-orbit in the set $\ParSet(G,\Delta_0,T,\mu,g_0,\delta)$. 
\qed 
\end{theorem}

\begin{remark} 
This theorem gives another parametrization of gradings for type $A_3=D_3$.
\end{remark}

\section{Appendix: special linear Lie algebras}

Type II gradings on $\mathfrak{sl}_n(\FF)$ (or $\mathfrak{psl}_n(\FF)$ if $\mathrm{char}\,\FF$ divides $n$) 
over an algebraically closed field $\FF$, $\mathrm{char}\,\FF\ne 2$, were constructed and classified 
in \cite{BK10,E10,EK13} by means of refining a grading on the associative algebra $M_n(\FF)$ with a suitable 
(degree-preserving) anti-automorphism of this algebra. In this paper, we have classified gradings on 
$\mathfrak{sl}_n(\Delta)$, $\Delta\in\{\RR,\HH\}$, in terms of gradings on the associative algebra 
$M_n(\Delta)\times M_n(\Delta)$ with involution of the second kind (Theorem \ref{th:AII_RH}). 
The purpose of this section is to establish a link between these two models.

Let $\FF$ be a field that is either algebraically closed with $\mathrm{char}\,\FF\ne 2$ or real closed (and hence 
$\mathrm{char}\,\FF=0$). Let $\cR$ be a central simple associative algebra over $\FF$ and 
let $\tilde{\cR}=\cR\otimes_\FF\KK$, where $\KK=\FF\times\FF$. Suppose $\tilde{\cR}$ is given a $G$-grading of Type II,
that is, its center $\KK$ is nontrivially graded: $\KK=\KK_e\oplus\KK_f$ where $\KK_e=\FF$, $\KK_f=\FF\zeta$, 
$\zeta=(1,-1)$, and $f\in G$ is an element of order $2$, referred to as the distinguished element.

\emph{Assume that there exists a character $\chi:G\to\FF^\times$ such that $\chi(f)=-1$.} 
(This is automatic if $\FF$ is algebraically closed, but note that $\chi$ cannot always be chosen to satisfy $\chi^2=1$.) 
Fix one such $\chi$ and consider its action on $\tilde{\cR}$ associated to the $G$-grading: $\psi(a):=\chi\cdot a=\chi(g)a$ 
for all $a\in\tilde{\cR}_g$ and $g\in G$. Since $\chi\cdot\zeta=-\zeta$, 
the automorphism $\psi$ interchanges the central idempotents $\varepsilon_1=(1,0)$ and $\varepsilon_2=(0,1)$, 
and hence we can use the restriction $\psi:\tilde{\cR}\varepsilon_1\to\tilde{\cR}\varepsilon_2$
to identify these simple components of $\tilde{\cR}$ with each other. 
Thus, we can identify $\tilde{\cR}$ with $\cR\times\cR$ such that $(x,y)\in\cR\times\cR$ corresponds to 
$x\varepsilon_1+\psi(y)\varepsilon_2\in\tilde\cR$, where we regard $\cR$ as a subalgebra of $\tilde{\cR}$ through
the canonical map $x\mapsto x\otimes 1$. Then the action of $\chi$ on $\cR\times\cR$ 
is given by $\chi\cdot(x,y)=(\psi^2(y),x)$.

Consider the coarsening of the grading on $\tilde{\cR}$ associated to the quotient map 
$G\to\overline{G}:=G/\langle f\rangle$,
that is, $\tilde{\cR}=\bigoplus_{\bar{g}\in\overline{G}}\tilde{\cR}_{\bar{g}}$ where 
$\tilde{\cR}_{\bar{g}}=\tilde{\cR}_g\oplus\tilde{\cR}_{gf}$ for all $g\in G$.
The idempotents $\varepsilon_1$ and $\varepsilon_2$ have degree $\bar{e}$, hence the simple components of $\tilde{\cR}$ 
are $\overline{G}$-graded and so is its subalgebra $\cR$. Moreover, the projections 
$\mathrm{pr}_1$ and $\mathrm{pr}_2$ of $\cR\times\cR$ onto $\cR$ 
are homomorphisms of $\overline{G}$-graded algebras. The action of $\chi$ can be used to recover the  
$G$-grading on $\tilde{\cR}$ from its coarsening as follows:
\[
\tilde{\cR}_g=\{a\in\tilde{\cR}_{\bar{g}}\mid\chi\cdot a=\chi(g)a\}
=\{(x,\chi(g)^{-1}x)\mid x\in\cR_{\bar{g}}\}.
\]
Since $\cR$ is simple, we can identify $\cR=\End_\cD(\cV)$ as a $\overline{G}$-graded algebra, 
for some graded-division algebra $\cD$ and a graded right $\cD$-module $\cV$. Let $\tilde{\cD}=\cD\times\cD$ 
and $\tilde{\cV}=\cV\times\cV$. We can refine the $\overline{G}$-gradings on $\tilde{\cD}$ and $\tilde{\cV}$ 
by setting $\tilde{\cD}_g:=\{(d,\chi(g)^{-1}d)\mid d\in\cD_{\bar{g}}\}$
and similarly for $\cV$. Then $\tilde{\cV}$ is a graded right $\tilde{\cD}$-module and 
$\tilde{\cR}\cong\End_{\tilde{\cD}}(\tilde{\cV})$ as a $G$-graded algebra.

\begin{remark}
The above is an example of the loop construction, which makes $G$-graded algebras and modules from 
$\overline{G}$-graded ones (see \cite{ABFP,EKloop}).
\end{remark}

It is easy to see that $\tilde{\cD}$ is a graded-division algebra with $\tilde{\cD}_e\cong\cD_e$,
the support $T$ of $\tilde{\cD}$ contains $f$, and the support of $\cD$ is 
$\overline{T}:=T/\langle f\rangle$. 
Moreover, the alternating bicharacter $\beta$ of $\tilde{\cD}$ has radical $\langle f\rangle$ and 
induces the nondegenerate bicharecter $\bar{\beta}$ of $\cD$ passing to the quotient modulo $\langle f\rangle$. 

Now, any involution of the second kind on $\cR\times\cR$ has the form 
$\tilde{\varphi}(x,y)=(\varphi(y),\varphi^{-1}(x))$ where $\varphi$ is an anti-automorphism of $\cR$.
Moreover, $\tilde{\varphi}$ is an involution of $\cR\times\cR$ as a $G$-graded algebra if and only if 
$\varphi$ is an anti-automorphism of $\cR$ as a $\overline{G}$-graded algebra and $\varphi^2(x)=\chi^2\cdot x$ 
for all $x\in\cR$. (Note that $\chi^2(f)=1$, so $\chi^2$ can be regarded as a character of $\overline{G}$
and therefore acts on the $\overline{G}$-graded algebra $\cR$; its action is the restriction of $\psi^2$.)
If this is the case, the $G$-grading on $\tilde{\cR}$ restricts to its Lie subalgebra 
$\cL=\mathrm{Skew}(\tilde{\cR},\tilde{\varphi})$, which is isomorphic to $\cR^{(-)}$ by means of $\mathrm{pr}_1$.
Since $\cL=\{(x,-\varphi^{-1}(x)) \mid x\in\cR\}$, the isomorphism $\mathrm{pr}_1\vert_\cL$ maps $\cL_g$ onto
\[
\cR_g=\{x\in\cR_{\bar{g}} \mid \varphi(x)=-\chi(g)x\}.
\] 
Note that $\cR=\bigoplus_{g\in G}\cR_g$ is a grading of the Lie algebra $\cR^{(-)}$ but not of the associative 
algebra $\cR$. It restricts to the Lie subalgebra $[\cR,\cR]$, which is a special linear algebra.
If $\mathrm{char}\,\FF>0$ then $\cR'=[\cR,\cR]$ may have a $1$-dimensional center; passing modulo the center, we 
obtain a grading on the corresponding projective special linear algebra, which is simple.

If $\FF$ is real closed then we know that the existence of $\tilde{\varphi}$ forces $T$ to be an elementary $2$-group. 
The same is true if $\FF$ is algebraically closed: because of $\varphi$, $\overline{T}$ must be an elementary $2$-group 
(see e.g. \cite[Lemma 2.50]{EK13}) and the same argument as in Subsection \ref{sse:consequences} shows that $f$ cannot 
be a square in $T$. Therefore, $\chi$ always takes values $\pm 1$ on $T$, and we can write 
$T=\overline{T}\times\langle f \rangle$ by identifying $\overline{T}$ with the subgroup $\{t\in T \mid \chi(t)=1\}$
(compare with Proposition 3.25 in \cite{EK13}, where $H$ stands for our $T$).
Hence, $\cD$ can be considered $G$-graded and $\tilde{\cD}$ can be identified with $\cD\otimes_\FF \KK$ 
as a $G$-graded algebra. If $\FF$ is algebraically closed, then $\cD_e=\FF$ and the isomorphism class of $\cD$
is determined by $(\overline{T},\bar{\beta})$. Using the same approach as in Section \ref{se:Lie}, 
we can obtain the following analog of Theorem \ref{th:AII_RH} for algebraically closed fields.

Let $\tilde{\cR}=\End_{\tilde{\cD}}(\tilde{\cV})$ where $\tilde{\cD}$ is determined by $(T,\beta)$ and $\tilde{\cV}$ 
is determined by a multiplicity function $\kappa:G/T\to\ZZ_{\ge 0}$ (with finite support). 
For each $(T,\beta)$, we fix a quadratic form $\eta$ on $T$ with polar form $\beta$ such that $\eta(f)=-1$.
This gives us an involution $\tilde{\varphi}_0$ of the second kind on $\tilde{\cD}$.
For a given $g_0\in G$, if $\kappa$ is admissible, that is, $\kappa(g_0^{-1}x^{-1})=\kappa(x)$ for all $x\in G/T$,
and $ \kappa(x) \equiv 0 \pmod 2 $ if $ x \in (G/T)_{g_0} $ and $ \eta(\tau(x)) = -1 $ 
(compare with Definition \ref{def:MultFunct}), then we have 
a nondegenerate hermitian form $\tilde{B}$ of degree $g_0$ on $\tilde{\cV}$, 
and all such hermitian forms are isomorphic (no signature functions in the algebraically closed case!).
The form $\tilde{B}$ gives us an involution on $\tilde{\cR}$, and the $G$-grading of $\cR$ induces 
an outer $G$-grading on the quotient of the Lie algebra $\mathrm{Skew}(\cR,\varphi)'$ modulo its center.
Denote this grading by $\Gamma_{\mathfrak{psl}}^{\mathrm{(II)}}(T,\beta,\eta,g_0,\kappa)$. 

\begin{theorem}\label{th:AII_ac}
Let $\FF$ be an algebraically closed field, $\mathrm{char}\,\FF\ne 2$.
Let $\cL$ be the simple Lie algebra of type $A_r$ for $r\ge 2$, namely, $\mathfrak{psl}_{r+1}(\FF)$.
If $\mathrm{char}\,\FF=3$, assume that $r\ge 3$. 
Then any outer $G$-grading on $\cL$ is isomorphic to 
$\Gamma_{\mathfrak{psl}}^{\mathrm{(II)}}(T,\beta,\eta,g_0,\kappa)$ where 
$r=\vert\kappa\vert \sqrt{\frac{1}{2}\vert T\vert}-1$.
Moreover, $\Gamma_{\mathfrak{psl}}^{\mathrm{(II)}}(T,\beta,\eta,g_0,\kappa)$
and $\Gamma_{\mathfrak{psl}}^{\mathrm{(II)}}(T',\beta',\eta',g_0',\kappa')$
are isomorphic if and only if
$ T = T'$, $ \beta = \beta' $ (hence $\eta=\eta'$), and
$(g_0,\kappa)$ and $(g_0',\kappa')$ are in the same $G$-orbit,
where $g\cdot(g_0,\kappa)=(g^{-2}g_0,g\cdot\kappa)$ and $(g\cdot\kappa)(x)=\kappa(g^{-1}x)$ for all $x\in G/T$. 
\qed
\end{theorem}

To compare this result with Theorem 3.53 in \cite{EK13}, we note that in that work the parameters of the grading 
are $(H,h,\beta,\kappa,\gamma,\mu_0,\bar{g}_0)$, where $H$ stands for our $T$, $h$ for $f$, $\beta$ for $\bar{\beta}$,
and $(\kappa,\gamma)$ for $\kappa$ ($\gamma$ refers to the support of the multiplicity function and $\kappa$ to the 
multiplicities of the elements of the support). Clearly, $f$ and $\bar{\beta}$ carry the same information as $\beta$.
As to the parameters $\mu_0\in\FF^\times$ and $\bar{g}_0\in\overline{G}$, they determine the $\varphi_0$-sesquilinear form 
$\cV\times\cV\to\cD$ that gives the anti-automorphism $\varphi$ on $\cR$. 
In \cite{EK13}, the involution $\varphi_0$ on $\cD$ was obtained by fixing a standard matrix realization of 
$\cD$ and setting $\varphi_0(X):=X^T$. (These are the involutions that correspond to quadratic forms $\bar{\eta}$ 
with polar form $\bar{\beta}$ and $\Arf(\bar{\eta})=1$.) 
Let $\iota$ be the involution of $\KK=\FF\times\FF$ that interchanges the two components 
(in other words, $\iota(\zeta)=-\zeta$). Then we can take $\varphi_0\otimes\iota$ on $\cD\otimes_\FF \KK$ 
as our $\tilde{\varphi}_0$. 
(The corresponding $\eta$ is related to $\bar{\eta}$ as follows: $\eta(t)=-\eta(tf)=\bar{\eta}(t)$ for all 
$t\in\overline{T}$.)
Any hermitian form $\tilde{B}$ must restrict to zero on each of the components $\cV\times\{0\}$ and $\{0\}\times\cV$ 
of $\tilde{\cV}$, so it is determined by its values $\tilde{B}((0,v),(w,0))$ for all $v,w\in\cV$. 
These values lie in $\cD\times\{0\}$ 
and hence we obtain a $\varphi_0$-sesquilinear form $B$ on $\cV$ by setting 
\[
(B(v,w),0):=\tilde{B}((0,v),(w,0)).
\]
It is straightforward to verify that $B$ is nondegenerate, has degree $\bar{g}_0=g_0\langle f\rangle$ with 
respect to the $\overline{G}$-grading, determines $\varphi$ by Equation \eqref{eq:GradAdj}, and has the following 
weak form of hermitian property: $\varphi_0(B(w,v))=\chi^{-1}(g_0)B(\chi^{-2}\cdot v,w)$ for all $v,w\in\cV$.
It follows that the parameter $\mu_0$ in \cite{EK13} is given by $\mu_0=\chi^{-1}(g_0)$. Since $\chi(f)=-1$, 
the element $g_0$ is determined by $\bar{g}_0$ and $\mu_0$.

\section*{Acknowledgments}

The authors want to thank Alberto Elduque for numerous fruitful discussions about gradings and mathematics in general.
The third author is also thankful to Professor Elduque for guidance as his thesis supervisor.

\end{document}